\documentclass{amsart}
\usepackage{amsmath}
\usepackage{amsthm}
\usepackage{amsfonts}
\usepackage{amssymb}
\usepackage{enumitem}
\usepackage[active]{srcltx}
\usepackage{upref}
\usepackage[colorlinks,citecolor=blue,linkcolor=blue]{hyperref}
\newcommand{\auth}[2]{{#2. #1}}
\newcommand{\art}[6]{{\sc #1, \rm #2, \it #3\/ \bf #4 \rm (#5), \mbox{#6}.}}
\newcommand{\arttoappear}[3]{{\sc #1, \rm #2, to appear in \it #3}}
\newcommand{\book}[3]{{\sc #1, \it #2, \rm #3.}}
\newcommand{\AND}{{\rm and }}

\newcommand{\reals}{\mathbf{R}}
\newcommand{\complex}{\mathbf{C}}

\newcommand{\norm}[1]{\left\|#1\right\|}

\newcommand{\paraa}[1]{\big(#1\big)}
\newcommand{\parac}[1]{\bigg(#1\bigg)}

\newcommand{\coker}{\operatorname{coker}}
\newcommand{\im}{\operatorname{im}}
\newcommand{\grad}{\nabla}
\newcommand{\dvg}{\mathop{{\rm div}}}
\newcommand{\bdy}{\partial}

\newcommand{\spacearound}[1]{\quad#1\quad}
\newcommand{\equivalent}{\spacearound{\Longleftrightarrow}}

\newtheorem{theorem}{Theorem}[section]
\newtheorem{corollary}[theorem]{Corollary}
\newtheorem{lemma}[theorem]{Lemma}
\newtheorem{proposition}[theorem]{Proposition}
\theoremstyle{definition}
\newtheorem{definition}[theorem]{Definition}
\newtheorem{example}[theorem]{Example}
\theoremstyle{remark}
\newtheorem{rem}[theorem]{Remark}
\numberwithin{equation}{section}

\newcommand{\tr}{\operatorname{tr}}
\newcommand{\A}{\mathcal{A}}
\newcommand{\B}{\mathcal{B}}
\newcommand{\K}{\mathcal{K}}
\newcommand{\Kt}{\widetilde{\K}}
\newcommand{\Kpsif}{\K_{\psi,f}}
\newcommand{\normp}[1]{\norm{#1}_{p}}
\newcommand{\normr}[1]{\norm{#1}_{r}}
\newcommand{\U}{\mathcal{U}}
\newcommand{\Ut}{\widetilde{\mathcal{U}}}
\newcommand{\normF}[1]{\norm{#1}_{\Frob}}
\DeclareMathOperator{\Frob}{Frob}
\newcommand{\normH}[1]{\norm{#1}_H}
\newcommand{\normV}[1]{\norm{#1}_V}
\newcommand{\normW}[1]{\norm{#1}_W}
\newcommand{\normY}[1]{\norm{#1}_Y}
\newcommand{\normSobU}[1]{\norm{#1}_{\SobU}}
\newcommand{\eps}{\varepsilon}
\newcommand{\Bplus}{B_{+}}
\newcommand{\Vplus}{V_{+}}

\newcommand{\Vt}{\widetilde{V}}
\newcommand{\Wt}{\widetilde{W}}
\newcommand{\ut}{\tilde{u}}
\newcommand{\gt}{\tilde{g}}
\newcommand{\htd}{\tilde{h}}
\newcommand{\st}{\tilde{s}}
\newcommand{\uh}{\hat{u}}
\newcommand{\MM}{\mathcal{M}}
\newcommand{\Om}{\Omega}
\newcommand{\Ftilde}{\widetilde{F}}
\renewcommand{\phi}{\varphi}
\renewcommand{\emptyset}{\varnothing}
{\catcode`p =12 \catcode`t =12 \gdef\eeaa#1pt{#1}}      
\def\accentadjtext#1{\setbox0\hbox{$#1$}\kern   
                \expandafter\eeaa\the\fontdimen1\textfont1 \ht0 }
\def\accentadjscript#1{\setbox0\hbox{$#1$}\kern 
                \expandafter\eeaa\the\fontdimen1\scriptfont1 \ht0 }
\def\accentadjscriptscript#1{\setbox0\hbox{$#1$}\kern   
                \expandafter\eeaa\the\fontdimen1\scriptscriptfont1 \ht0 }
\def\accentadjtextback#1{\setbox0\hbox{$#1$}\kern       
                -\expandafter\eeaa\the\fontdimen1\textfont1 \ht0 }
\def\accentadjscriptback#1{\setbox0\hbox{$#1$}\kern     
                -\expandafter\eeaa\the\fontdimen1\scriptfont1 \ht0 }
\def\accentadjscriptscriptback#1{\setbox0\hbox{$#1$}\kern 
                -\expandafter\eeaa\the\fontdimen1\scriptscriptfont1 \ht0 }
\def\itoverline#1{{\mathsurround0pt\mathchoice
        {\rlap{$\accentadjtext{\displaystyle #1}
                \accentadjtext{\vrule height1.593pt}
                \overline{\phantom{\displaystyle #1}
                \accentadjtextback{\displaystyle #1}}$}{#1}}
        {\rlap{$\accentadjtext{\textstyle #1}
                \accentadjtext{\vrule height1.593pt}
                \overline{\phantom{\textstyle #1}
                \accentadjtextback{\textstyle #1}}$}{#1}}
        {\rlap{$\accentadjscript{\scriptstyle #1}
                \accentadjscript{\vrule height1.593pt}
                \overline{\phantom{\scriptstyle #1}
                \accentadjscriptback{\scriptstyle #1}}$}{#1}}
        {\rlap{$\accentadjscriptscript{\scriptscriptstyle #1}
                \accentadjscriptscript{\vrule height1.593pt}
                \overline{\phantom{\scriptscriptstyle #1}
                \accentadjscriptscriptback{\scriptscriptstyle #1}}$}{#1}}}}

\def\vint{\mathop{\mathchoice%
          {\setbox0\hbox{$\displaystyle\intop$}\kern 0.22\wd0%
           \vcenter{\hrule width 0.6\wd0}\kern -0.82\wd0}%
          {\setbox0\hbox{$\textstyle\intop$}\kern 0.2\wd0%
           \vcenter{\hrule width 0.6\wd0}\kern -0.8\wd0}%
          {\setbox0\hbox{$\scriptstyle\intop$}\kern 0.2\wd0%
           \vcenter{\hrule width 0.6\wd0}\kern -0.8\wd0}%
          {\setbox0\hbox{$\scriptscriptstyle\intop$}\kern 0.2\wd0%
           \vcenter{\hrule width 0.6\wd0}\kern -0.8\wd0}}%
          \mathopen{}\int}

\newcommand{\ga}{\gamma}
\newcommand{\p}{{$p\mspace{1mu}$}}
\newcommand{\Np}{N^{1,p}}
\newcommand{\hNp}{\widehat{N}^{1,p}}
\DeclareMathOperator{\Sob}{Sob}
\newcommand{\SobU}{{\Sob(\U)}}
\newcommand{\setm}{\setminus}
\newcommand{\C}{\mathbf{C}}
\newcommand{\R}{\mathbf{R}}
\newcommand{\Q}{\mathbf{Q}}
\newcommand{\al}{\alpha}
\newcommand{\alp}{\alpha}
\newcommand{\imp}{\ensuremath{\Rightarrow} }
\DeclareMathOperator{\varRe}{Re}
\renewcommand{\Re}{\varRe}
\DeclareMathOperator{\varIm}{Im}
\renewcommand{\Im}{\varIm}
\DeclareMathOperator{\Ray}{Ray}

\newenvironment{ack}{\medskip{\it Acknowledgement.}}{}

\title[An axiomatic approach to gradients]
{An axiomatic approach to gradients with\\ 
applications to  Dirichlet and  obstacle\\ problems
 beyond function spaces}

\author{Joakim Arnlind, Anders Bj\"orn and Jana Bj\"orn}
\address{Department of Mathematics\\
Link\"oping University\\
SE-581 83 Link\"oping\\
Sweden}
\email{joakim.arnlind@liu.se, anders.bjorn@liu.se and jana.bjorn@liu.se}

\thanks{}



\begin{document}

\begin{abstract}
  We develop a framework for studying variational
  problems in Banach spaces with respect to gradient relations, which
  encompasses many of the notions of generalized gradients that appear
  in the literature. We stress the fact that our approach is not
  dependent on function spaces and therefore applies equally well to
  functions on metric spaces as to operator algebras. In particular,
  we consider analogues of Dirichlet and obstacle problems, as well as
  first eigenvalue problems, and formulate conditions for the
  existence of solutions and their uniqueness. Moreover, we
  investigate to what extent a lattice structure may be introduced on
  (ordered) Banach spaces via a norm-minimizing variational problem. A
  multitude of examples is provided to illustrate the versatility of
  our approach.
\end{abstract}

\maketitle

\section{Introduction}

\noindent In the classical theory of partial differential equations,
one explores the existence of solutions (and their regularity) by
extending spaces of differentiable functions to include functions with
only a weak notion of derivative. Introducing $L^p$-spaces and Sobolev
spaces has the advantage that one may 
exploit the
completeness of these spaces in order to find weak solutions of
differential equations.  In doing so, one is forced to work with
equivalence classes of functions, rather than single functions, and
the classical value of a function at a point is, for some purposes,
simply not relevant anymore.  Consequently, one tends to use Banach
space techniques to reach the desired results. In particular, when
extending the theory to functions on more general spaces, it becomes
apparent that abstract methods are useful as classical techniques may
not be applicable.

Consider the Dirichlet problem for harmonic functions, i.e.\
to find a harmonic function with given boundary values in a 
bounded domain $\Om$
in $\R^n$. This problem can equivalently be reformulated as
finding the minimizer of the energy integral
\begin{align}
  \label{eq:dirichlet.energy}
  \norm{\nabla u}^2_{L^2(\Om)}= \int_\Om |\nabla u|^2\,dx,  
\end{align}
over all sufficiently smooth functions with given boundary values. In
this note, we aim to give an axiomatic approach to such problems
starting from a quite general notion of gradient, assuming only a weak
form of linearity. Many particular examples of gradients, such as weak
gradients, upper gradients in metric spaces, Haj\l asz gradients and
algebraic derivations, fall into this class.  We shall also consider
gradients with no relation to derivatives (cf.\
Section~\ref{sect-ex-2}), as well as examples which come from 
higher-order differential operators, such as the Laplacian and
biLaplacian (cf.\
Section~\ref{sec:higher.order.operators}). It deserves to be pointed
out that the framework we develop depends neither on function spaces
nor on the commutativity of multiplication and, therefore, applies
equally well to noncommutative settings, such as operator algebras.

We start by introducing an abstract
notion of gradient relation and define a Sobolev space
based on it.  We show that, under minimal assumptions, this
generalized Sobolev space is always a Banach space and that functions
therein possess a unique minimal gradient.  In
Theorem~\ref{thm:dirichlet.problem} we formulate sufficient
conditions for the existence of solutions to the Dirichlet problem
with respect to this minimal gradient in analogy with
\eqref{eq:dirichlet.energy}. Furthermore, in
Proposition~\ref{prop:u.gu.linear.unique.solution} we give a condition
for the solution to be unique.

In addition to the Dirichlet problem, we also consider the obstacle
problem as well as the first eigenvalue problem (strictly speaking the
existence of minimizers for the Rayleigh quotient,
cf.\ Theorem~\ref{thm-solve-rayleigh}).  To solve the obstacle
problem we reformulate it as a Dirichlet problem, and we can thus use
the Dirichlet problem theory to solve the obstacle problem.  Already
here one can see the power of our abstract approach, as one can rarely
consider obstacle problems as special cases of
Dirichlet problems in more traditional
situations (cf.\ Remark~\ref{rem:poincare.sets}).  A prominent
role in the minimization problems above is played by
\emph{Poincar\'e sets}, i.e.\ subsets $\K$ of the abstract Sobolev
space which support 
a generalized Poincar\'e inequality:
\begin{align*}
  \norm{u}\leq C\norm{\nabla u}, \quad u\in\K.
\end{align*}
Such sets provide natural domains when considering variational
problems in the context of gradient relations.

Finally, inspired by the fact that the pointwise maximum of two
functions minimizes the $L^p$-norm among all
functions which majorize both 
functions,
we investigate the possibility of defining the maximum (as well as the
minimum) of two elements in a Banach space via a minimization problem
(cf.\ Propositions~\ref{prop:existence.maximum}
and~\ref{prop:min.via.minimization}). Furthermore, we formulate
necessary and sufficient conditions for the existence of least upper
(resp.\ greatest lower) bounds (cf.\ Theorem~\ref{thm-char-lattice}).

The paper is organized as follows. Section~\ref{sect-grad-spcs}
introduces the generalized concept of gradient that we shall be
studying, as well as the corresponding concept of gradient
space and the associated Sobolev space. 
These objects are the basic ingredients of our
analysis. Sections~\ref{subsect-DP} and~\ref{sec-obstacle-problem}
introduce the Dirichlet and obstacle problems together with the concept
of preordered gradient spaces. It is shown that, under certain
conditions, solutions of the Dirichlet and obstacle problems exist. In
Section~\ref{sec:rayleigh.quotient} we consider the Rayleigh quotient
which, in classical analysis, is related to finding the first positive
eigenvalue of the Laplace operator. Also here, a minimizer can be
found under certain assumptions.

In Section~\ref{sect-order} we investigate the possibility of defining
a least upper bound via a minimization problem. While the least upper
bound will in general not exist,
we show that a minimizer (in
norm) exists, retaining some of the properties of a least upper
bound. Sections~\ref{sect-ex-1} and~\ref{sect-ex-2} are devoted to
showing that many situations can be treated in a unified way within our
framework. The examples include both classical function spaces and
functions on metric spaces (together with their appropriate concepts
of gradients) as well as noncommutative examples such as spaces of
matrix algebras, operator algebras and operator-valued functions.

\begin{ack}
The  authors were supported by the Swedish Research Council.
\end{ack}

\section{Gradient spaces} 
\label{sect-grad-spcs}

\noindent When moving away from the realm of differentiable functions
defined on vector spaces, one is lead to introduce several (in
general, different) concepts of a derivative, or
gradient. For instance, one may consider weak derivatives on
$\R^n$ or upper gradients on metric spaces. In this section, we
will introduce a very weak abstract
notion of gradient, which encompasses many
of the situations one would like to study. There is actually no
mathematical reason for using the name gradient (as the assumptions
only include a weak form of linearity), but we have chosen to keep the
terminology both for historical reasons and in view of many
applications. Moreover, we shall introduce a corresponding pair of Banach
spaces, which provides a link between the gradient and the analytic
structure of the normed space.

Let us start by introducing the concept of a gradient, in the form of
a relation on the Cartesian product of two vector spaces.

\begin{definition}\label{def:upper.gradient.relation}
  Let $\Vt$ and $\Wt$ be vector spaces over $\reals$ or $\complex$ (not
  necessarily the same for $\Vt$ and $\Wt$). A 
  \emph{gradient relation}
  is a relation $R\subseteq \Vt\times \Wt$ such that
  \begin{enumerate}[label=(G\arabic{*}),ref=(G\arabic{*}),leftmargin=15mm]
  \item if $(u,g)\in R$ and $(u',g')\in R$ then $(u+u',g+g')\in R$,\label{g:sum}
  \item if $(u,g)\in R$ and $\alpha>0$ then $(\alpha u,\alpha g)\in R$.\label{g:mult.scalar}
  \end{enumerate}
  We say that \emph{$g$ is a gradient of $u$} if $(u,g)\in R$.
\end{definition}

\noindent To be able to address analytic questions, we introduce
subspaces $V$ and $W$, of $\Vt$ and $\Wt$, which have the structure of
Banach spaces.  Note that 
property~\ref{gs:upper.gradient.limit}
below is important because it
connects the gradient relation to the analytic structure 
of the space.

\begin{definition}\label{def:gradient.space}
  A \emph{gradient space} $\U=\paraa{V,\Vt,W,\Wt,R}$ consists
  of two vector spaces $\Vt$ and $\Wt$ together with a gradient relation
  $R\subseteq \Vt\times \Wt$ and linear subspaces $V\subseteq \Vt$ and
  $W\subseteq \Wt$ such that
  \begin{enumerate}[label=(GS\arabic{*}),ref=(GS\arabic{*}),leftmargin=15mm]
  \item $V$ is a reflexive Banach space,\label{gs:V.reflexive}
  \item $W$ is a reflexive and strictly convex Banach space,\label{gs:W.reflexive.convex}
  \item if $(u,g)\in R$ with $u\in V$ and $g\in W$, then there exists
    $g'\in W$ such that $(-u,g')\in R$,\label{gs:minus.gradient}
  \item if $u,u_i\in V$ and $g,g_i\in W$ with $(u_i,g_i)\in R$,
    for $i=1,2,\ldots$, are  such that $\normV{u-u_i}\to 0$ and
    $\normW{g-g_i}\to 0$, then $(u,g)\in R$,\label{gs:upper.gradient.limit}
  \end{enumerate}
  where $\normV{\cdot}$ and $\normW{\cdot}$ denote the norms of $V$
  and $W$, respectively.
\end{definition}

\noindent
Note that for classical derivatives and gradients, as well
as for weak gradients, one has $g'=-g$ in \ref{gs:minus.gradient},
while for (weak) upper gradients and Haj\l asz gradients
(see Sections~\ref{sec:pweak.upper.gradient}--\ref{sec:Hajlasz.gradient}) one
has $g'=g$.
In \ref{gs:minus.gradient} we allow for even more general situations,
and do not even require a uniform bound of the form
$\normW{g'} \le C \normW{g}$.

Just as for $L^p$-spaces, it is natural to introduce a
subspace of $V$, consisting of elements which have a gradient in
$W$. 

\medskip
\noindent
\emph{Assume for the rest of this section that $\U=(V,\Vt,W,\Wt,R)$
is a gradient space.}

\begin{definition}
The set
  \begin{align*}
    \SobU := \{u\in V: (u,g)\in R \text{ for some }  g\in W\}
  \end{align*}
is called the \emph{Sobolev space of the gradient space $\U$}.
\end{definition}

\begin{lemma} \label{lem-Sob-vector-space} 
  If $\SobU\neq\emptyset$, then $\SobU$ is
  a vector space over $\reals$.
\end{lemma}

\begin{proof}
As $\SobU\neq\emptyset$, there exists $u\in\SobU$ which, by
  \ref{gs:minus.gradient}, implies that $-u\in\SobU$.  From \ref{g:sum}
  it follows that if $u,u'\in\SobU$ then $u+u'\in\SobU$, since 
  $V$ and $W$ are vector spaces. 
  In particular, this implies that $u+(-u)=0\in\SobU$.
  Finally, from \ref{g:mult.scalar} and \ref{gs:minus.gradient} (and
  the fact that $0\in\SobU$) it follows that if $u\in\SobU$ and
  $\alpha\in\reals$ then $\alpha u\in\SobU$.
\end{proof} 

\begin{lemma}\label{lemma:DpU.nonempty.zero.gradient}
  If $\SobU\neq\emptyset$, then $(0,0)\in R$.
\end{lemma}

\begin{proof}
  If $\SobU\neq\emptyset$, 
then $0 \in \SobU$, by Lemma~\ref{lem-Sob-vector-space}, and thus there
is $g \in W$ such that $(0,g) \in R$.
Hence,  $(0,g/k)\in R$ for $k=1,2,\ldots$, by \ref{g:mult.scalar}.
Since $\normW{g/k}=\normW{g}/k\to 0$, as $k \to \infty$, 
it follows from 
  \ref{gs:upper.gradient.limit} that $(0,0)\in R$.
\end{proof}

\noindent The following results (Lemma~\ref{lem:weak.strong.closure},
Corollary~\ref{cor:weak.convex.combination} and 
Lemma~\ref{lemma:bounded.seq.upper.gradient}) are fairly standard, but we
have chosen to repeat them here adjusted to our setting. They constitute a
set of very useful technical results, and concern the possibility of
constructing strongly convergent sequences from weakly convergent
sequences, by using convex combinations.

\begin{lemma}[Theorem~3.12 in Rudin~\cite{r:functionalAnalysis}]\label{lem:weak.strong.closure}
  If $E$ is a convex subset of a locally convex space, 
   then the weak
  closure of $E$ equals the\/ \textup{(}strong\/\textup{)} closure of $E$.
\end{lemma}

\noindent
The following results can be found in the literature under the name 
Mazur's lemma, with the slight difference that the linear combinations
are usually taken starting from $j=1$. 
To prove Mazur's lemma in the form given here, an iterative argument 
is then needed.
We therefore provide a full proof of the result, tailored to our needs.
Note that the sums below start at $j=i$, which
 will be important when we apply this result.
Recall that a \emph{convex combination}  
is a linear combination with nonnegative coefficients 
summing up to one, such that
only finitely many 
of them are positive.

\begin{corollary}\label{cor:weak.convex.combination}
  Let $\{u_i\}_{i=1}^\infty$ be a weakly convergent sequence with\/ 
  weak 
  limit $u$ in a normed vector space. Then there exist convex
  combinations
  \begin{align*}
    \tilde{u}_i = \sum_{j=i}^{N_i}\alpha_{ij}u_j
    \quad \text{with }
\alpha_{ij}\geq 0\textrm{ and }\sum_{j=i}^{N_i}\alpha_{ij}=1,
  \end{align*}
  such that the sequence $\{\tilde{u}_i\}_{i=1}^\infty$ is strongly
  convergent to $u$.  
\end{corollary}

\begin{proof}
Let $E_i$ denote the convex hull of
  $\{u_i,u_{i+1},\ldots\}$ (i.e.\ the set of its convex
  combinations), and let $\itoverline{E}_i^{\mspace{1mu}w}$ 
and $\itoverline{E}_i$ denote the weak
  and strong closures of $E_i$, respectively. 
By assumption,
  $u\in \itoverline{E}_i^{\mspace{1mu}w}$ for $i=1,2,\ldots$.  Furthermore, it follows from
  Lemma~\ref{lem:weak.strong.closure} that $u\in \itoverline{E}_i$ for
  $i=1,2,\ldots$, which implies that there exists, for each $i$, 
$\tilde{u}_i\in E_i$ such that
  \begin{align*}
    \norm{\tilde{u}_i-u}<\frac{1}{i}.
  \end{align*}
Thus $\tilde{u}_i \to u$ strongly, 
and as
$\tilde{u}_i\in E_i$, it 
is a convex combination of the elements $u_i,u_{i+1},u_{i+2},\ldots$.
\end{proof}

\begin{lemma}\label{lemma:bounded.seq.upper.gradient}
  Let 
  $\{u_i\}_{i=1}^\infty\subseteq V$ and
  $\{g_i\}_{i=1}^\infty\subseteq W$ be bounded
  sequences such that $(u_i,g_i)\in R$ for $i=1,2,\ldots$. Then there
  exist  
convex combinations
  \begin{align*}
    \tilde{u}_i = \sum_{j=i}^{N_i}\alpha_{ij}u_j\quad\textrm{and}\quad
    \tilde{g}_i = \sum_{j=i}^{N_i}\alpha_{ij}g_j,
  \end{align*}
  with limits $u:=\lim_{i\to\infty}\tilde{u}_i$ and
  $g:=\lim_{i\to\infty}\tilde{g}_i$, such that $(u,g)\in R$.
\end{lemma}

\begin{proof}
  Since $\{u_i\}_{i=1}^\infty$ is a bounded sequence and $V$ is
  reflexive (by property \ref{gs:V.reflexive}), 
  Banach--Alaoglu's theorem shows that there exists a weakly convergent
  subsequence, with some weak limit $u$, which, by a slight abuse of
  notation, we shall also denote by $\{u_i\}_{i=1}^\infty$. Moreover,
  we assume that $\{g_i\}_{i=1}^\infty$ denotes the corresponding (not
  necessarily weakly convergent) subsequence of gradients. From
  Corollary~\ref{cor:weak.convex.combination} it follows that there
  exist convex combinations
  \begin{align*}
     \hat{u}_i = \sum_{j=i}^{\widehat{N}_i}\widehat{\alpha}_{ij}u_j
  \end{align*}
  such that the sequence $\{\hat{u}_i\}_{i=1}^\infty$ converges
  (strongly) to $u$. The corresponding linear combinations of
  gradients 
  \begin{align*}
    \hat{g}_i = \sum_{j=i}^{\widehat{N}_i}\widehat{\alpha}_{ij}g_j
  \end{align*}
  fulfill $(\hat{u}_i,\hat{g}_i)\in R$ for $i=1,2,\ldots$ (which
  follows from \ref{g:sum} and \ref{g:mult.scalar}). Now, since
  $\{\hat{g}_i\}_{i=1}^\infty$ is still a bounded sequence
  and $W$ is reflexive, Banach--Alaoglu's theorem again shows that there
  exists a weakly convergent subsequence
  $\{\hat{g}_{j_k}\}_{k=1}^\infty$ with some weak limit $g$. Using
  Corollary~\ref{cor:weak.convex.combination} once more, one finds convex
  combinations 
  \begin{align*}
    \tilde{g}_i = \sum_{k=i}^{N_i}\widetilde{\alpha}_{ik}\hat{g}_{j_k}
  \end{align*}
  such that the sequence $\{\tilde{g}_i\}_{i=1}^\infty$ converges
  (strongly) to $g$. The corresponding linear combinations
  \begin{align*}
    \tilde{u}_i = \sum_{k=i}^{N_i}\widetilde{\alpha}_{ik}\hat{u}_{j_k}
  \end{align*}
  fulfill $(\tilde{u}_i,\tilde{g}_i)\in R$ for $i=1,2,\ldots$. Finally,
  let us show that $\{\tilde{u}_i\}_{i=1}^\infty$ converges to
  $u$. Namely, since the sequence $\{\hat{u}_i\}_{i=1}^\infty$
  converges to $u$, the subsequence 
$\{\hat{u}_{j_k}\}_{k=1}^\infty$
  converges to $u$. Furthermore,  
  \begin{align*}
    \normV{\tilde{u}_i-u} =
    \biggl\|\sum_{k=i}^{N_i}\widetilde{\alpha}_{ik}\hat{u}_{j_k}-\sum_{k=i}^{N_i}
           \widetilde{\alpha}_{ik}u\biggr\|\leq
    \sum_{k=i}^{N_i}\widetilde{\alpha}_{ik}\normV{\hat{u}_{j_k}-u},
  \end{align*}
  which shows that $\lim_{i\to\infty}\tilde{u}_i=u$ since
  $\lim_{i\to\infty}\hat{u}_{j_k}=u$. Hence, we have found two
  (strongly) convergent sequences $\{\tilde{u}_i\}_{i=1}^\infty$ and
  $\{\tilde{g}_i\}_{i=1}^\infty$ (given as convex combinations of
  the original bounded sequences), with $\lim_{i\to\infty}\tilde{u}_i=u$
  and $\lim_{i\to\infty}\tilde{g}_i=g$, such that
  $(\tilde{u}_i,\tilde{g}_i)\in R$. By property
  \ref{gs:upper.gradient.limit} it follows that $(u,g)\in R$.
\end{proof}

\noindent For the Sobolev space $\SobU$, it is natural to introduce
the norm
\begin{align*}
  \normSobU{u} = \normV{u} + \inf_{(u,g)\in R} \normW{g},
\end{align*}
where the infimum is taken over all gradients $g$ of $u$.
Since $g+g'$ is a gradient of $u+u'$ (where $g$ and $g'$ are gradients
of $u$ and $u'$, respectively) 
it is clear that the triangle inequality is
fulfilled, making $\normSobU{\cdot}$ a norm on $\SobU$. 

\begin{theorem}
The Sobolev space $\SobU$  
 is a Banach space.
\end{theorem}

\begin{proof}
  It is easy to check that $(\SobU,\normSobU{\cdot})$ is a normed
  space, so the only thing remaining is to show that $\SobU$ is
  complete with respect to the given norm.  Let $\{u_j\}_{j=1}^\infty$
  be a Cauchy sequence in $\SobU$, i.e.\ for every $n=0,1,\ldots$,
  there exists an index $k_n$ such that 
  \[
    \|u_j-u_k\|_\SobU < 2^{-n} \quad \text{whenever } j,k\ge k_n.
  \]
  In particular, $\|u_j-u_k\|_V < 2^{-n}$ and
  there exists a gradient $g_{jk}$ of $u_j-u_k$ such that
  $\|g_{jk}\|_W < 2^{-n}$ whenever $j,k\ge k_n$.  It follows directly
  that $\{u_j\}_{j=1}^\infty$ is a Cauchy sequence in $V$, and thus
  has a limit $u:=\lim_{j\to \infty} u_j$, as $V$ is complete.

Our next aim is to show that $u_j\to u$ also in $\SobU$, and we shall
proceed by constructing gradients for $u_{k_l}-u$
that tend to zero in $W$ as
$l\to\infty$. Thus, for every $j$ and $l$ let
\[
s_{lj} := u_{k_l} - u_{k_{l+j}} = \sum_{i=1}^j (u_{k_{l+i-1}} - u_{k_{l+i}})
\quad \text{and} \quad
h_{lj} = \sum_{i=1}^j g_{k_{l+i-1},k_{l+i}},
\]
from which it follows that $h_{lj}$ is a gradient of
$s_{lj}$. Moreover, we set 
\[
s_l:=\lim_{j\to\infty} s_{lj}=u_{k_l}-u.
\]
As 
\[
\|s_{lj}\|_V \le \sum_{i=1}^j 2^{-(l+i-1)} < 2^{1-l}
\quad \text{and} \quad
\|h_{lj}\|_W \le \sum_{i=1}^j 2^{-(l+i-1)} < 2^{1-l},
\]
the sequences $\{s_{lj}\}_{j=1}^{\infty}$ and
$\{h_{lj}\}_{j=1}^\infty$ are bounded. 
An application of Lemma~\ref{lemma:bounded.seq.upper.gradient}
provides us with convex combinations
\[
\st_{li} = \sum_{j=i}^{N_i} \al_{ij}s_{lj}\quad\textrm{and}\quad
\htd_{li} = \sum_{j=i}^{N_i} \al_{ij}h_{lj}
\]
such that $\st_{li} \to s_l$ in $V$ and $\htd_{li} \to h_l$
in $W$, as $i\to\infty$, and $(s_l,h_l)\in R$.
Furthermore,
\[
\|h_l\|_W = \lim_{i\to\infty} \|\htd_{li}\|_W \le \lim_{i\to\infty} \|h_{li}\|_W
< 2^{1-l}.
\]
Finally, for every $k\ge k_l$, we see that $h_l+g_{k_l,k}$ is a gradient of
$u_k-u=(u_k-u_{k_l})+s_l$ and 
\[
\|h_l+g_{k_l,k}\|_W \le \|h_l\|_W + \|g_{k_l,k}\|_W < 2^{1-l} + 2^{-l}
\to 0, \quad \text{as } l \to \infty.
\]
Letting $l\to \infty$ we conclude that
\[
\|u_k-u\|_\SobU \le \|u_k-u\|_V + \|h_l+g_{k_l,k}\|_W \to 0,
\quad \text{as } k\to\infty.
\qedhere
\]
\end{proof}

\noindent In general, there can be many gradients of an element $u\in V$, 
see e.g.\ Sections~\ref{sec:pweak.upper.gradient}--\ref{sec-grad-from-PI},
but we shall mainly be interested in the minimal one (in the
following sense).

\begin{definition}
  Let 
  $u\in V$. An element $g_u\in W$ 
is a \emph{minimal gradient of $u$} if $(u,g_u)\in R$ and
  \[
    \normW{g_u}\leq \normW{g} \quad \text{for all } 
g\in W \text{ such that } (u,g)\in R.
  \]
\end{definition}

\begin{theorem}\label{thm:DpU.minimal.gradient}
  Every element $u\in\SobU$ has a unique minimal gradient.
\end{theorem}

\noindent
We will denote the minimal gradient of $u$ by $g_u$. Note that, if
$\alpha\ge 0$, then $g_{\alpha u}=\alpha g_u$ since
$\alp g_u$ is a gradient of $g_{\alp u}$ and (when $\alp >0$)
\begin{align*}
  \normW{g_u} = \normW{g_{\alpha^{-1}\alpha u}}
  \leq \alpha^{-1}\normW{g_{\alpha u}}\leq
  \alpha^{-1}\alpha\normW{g_u}
  =\normW{g_u},
\end{align*}
and the uniqueness in Theorem~\ref{thm:DpU.minimal.gradient}
shows that $g_{\alpha u}=\alpha g_u$.

\begin{proof}
  Let $I=\inf_{g}\normW{g}$, where the infimum is over all
  gradients of $u$ in $W$, and let $\{g_j\}_{j=1}^\infty$ 
  be a minimizing sequence, i.e.
  \begin{align*}
    \lim_{j\to\infty}\normW{g_j} = I,
  \end{align*}
  where $g_j$ is a gradient of $u$ for $j=1,2,\ldots$. The
  minimizing sequence is clearly bounded, and from 
  Lemma~\ref{lemma:bounded.seq.upper.gradient} (with $u_i=u$ for
  $i=1,2,\ldots$) it follows that there exist convex combinations
  $\gt_i=\sum_{j=i}^{N_i} \alp_{ij} g_j$ 
  converging to some $g\in W$ with $(u,g)\in R$.
 Since $g$
  is a gradient of $u$ one has $I\leq \normW{g}$, and 
thus 
  \begin{align*}
   I \le  \normW{g} = \lim_{i \to \infty} \normW{\gt_i}
   \leq\limsup_{j\to\infty}\normW{g_{j}}=I,
  \end{align*}
  which implies that $\normW{g}=I$. 

  Let us now prove uniqueness. If $g_1$ and $g_2$ are two minimal
  gradients of $u$ then $h=\frac{1}{2}(g_1+g_2)$ is also a gradient of $u$,
  which implies that
  \begin{align*}
    I\leq\normW{h}\leq \tfrac{1}{2}(\normW{g_1}+\normW{g_2})=I
  \end{align*}
  and we conclude that $h$ is also a minimal gradient of
  $u$. Hence
  \begin{align*}
    \normW{g_1}=\normW{g_2}=\normW{\tfrac{1}{2}(g_1+g_2)}
  \end{align*}
  and since $W$ is assumed to be a strictly convex space, it
  follows that $\normW{g_1-g_2}=0$, which proves that the minimal
  gradient is unique.
\end{proof}

\section{The Dirichlet problem}
\label{subsect-DP}

\noindent
\emph{Assume in  this section that $\U=(V,\Vt,W,\Wt,R)$
is a gradient space.}

\medskip

\noindent The classical Dirichlet problem for
harmonic (or \p-harmonic) functions can be formulated as follows:
subject to given boundary conditions in a domain,
one tries to find a weakly differentiable function 
whose (weak) gradient has minimal norm among all
such functions satisfying the boundary conditions.

In the setting of gradient spaces, we shall formulate the problem in
the following way. Let $\K_0$ be a subset of $\SobU$ (which, in the
classical setting, corresponds to the set of functions in a domain
with zero boundary values). Given $f\in\SobU$ we set
\begin{align*}
  \K_f=\K_0+f=\{v\in \SobU:v-f\in \K_0\},
\end{align*}
and think of $\K_f$ as the analogue of the set of functions which are
equal to $f$ on the boundary of the domain. A solution of the
Dirichlet problem with respect to $\K_f$ is then given by an
element $u\in\K_f$ such that
\begin{align*}
  \normW{g_u} = \inf_{v\in\K_f}\normW{g_v},
\end{align*}
where $g_u$ and $g_v$ denote the minimal gradients of $u$ and $v$,
respectively. As stated, the Dirichlet problem does not have enough
analytic structure to ensure the existence of a solution in the
general case. Therefore, a Poincar\'e inequality is introduced, which
allows one to obtain a bound on the norm of an element in $\K_0$ in
terms of the norm of its minimal gradient.

\begin{definition}
  A
  \emph{Poincar\'e set} is a subset $A\subseteq\SobU$ for which there
  exists a constant $C>0$ such that for all $u\in A$,
  \begin{align*}
    \normV{u}\leq C\normW{g}
  \end{align*}
  for all $g\in W$ such that $(u,g)\in R$.
\end{definition}

\noindent Under the assumption that $\K_0$ is a (closed and convex)
Poincar\'e set one can obtain the existence of a solution to the
Dirichlet problem. However, in general the solution will not be
unique, see Examples~\ref{ex-grad-max-1} and~\ref{ex-grad-max-2}.

\begin{theorem}\label{thm:dirichlet.problem}
  For any nonempty closed convex Poincar\'e set
  $\K_0\subseteq{\SobU}$ and $f\in\SobU$, the Dirichlet problem with
  respect to $\K_f=\K_0+f$ has at least one solution. 
Moreover, if $u_1$ and $u_2$ are solutions to
  the Dirichlet problem, then $g_{u_1}=g_{u_2}$.
\end{theorem}

\begin{rem}\label{rmk-bdd-minim-seq-no-PI}
The proof shows that if the minimizing sequence in the Dirichlet problem
can a priori be assumed to be bounded in $V$, then the assumption that
$\K_0$ is a Poincar\'e set can be omitted in 
Theorem~\ref{thm:dirichlet.problem}.
This observation was e.g.\ used in 
Bj\"orn--Bj\"orn~\cite[Theorem~5.13]{BBnonopen}  
to prove the existence of capacitary minimizers for arbitrary bounded condensers.
See Section~5 therein for examples and further discussion on the role of Poincar\'e 
inequalities in the Dirichlet problem.
\end{rem}

\begin{proof}
  Let $\{u_j\}_{j=1}^\infty\subseteq\K_f$ be a minimizing sequence, i.e. 
  \[
    \lim_{j\to\infty}\normW{g_j} = \inf_{v\in\K_f}\normW{g_v} =: I,
  \]
  where $g_j$ denotes the minimal gradient of $u_j$. Clearly,
  the sequence $\{g_j\}_{j=1}^\infty$ is bounded. Since $\K_0$ is a Poincar\'e
  set there exists a $C>0$ such that
  \begin{align*}
    \normV{u_j-f}\leq C\normW{g_{u_j-f}},
  \end{align*}
  and, by using \ref{gs:minus.gradient}, there
  exists a $g'\in W$ such that
  \begin{align*} 
    \normV{u_j-f}\leq C\normW{g_{u_j-f}}\leq C'(\normW{g_j}+\normW{g'}),
  \end{align*}  
  which implies that $\{u_j\}_{j=1}^\infty$ is also a bounded sequence. 
  From Lemma~\ref{lemma:bounded.seq.upper.gradient} 
  it follows that there exist
  convex combinations $\{\tilde{u}_i\}_{i=1}^\infty$ and
  $\{\tilde{g}_i\}_{i=1}^\infty$, converging to some functions
  $u$ and $g$,
  respectively, with $(u,g)\in R$. As
  $\K_0$ is convex we have 
 $\tilde{u}_i\in\K_f$ and,
  furthermore, since $\K_0$ is  closed, it follows
  that $u\in\K_f$. It remains to show that $g_u$ is indeed a minimizer, which is
  easily seen from
  \begin{align*}
    I\leq \normW{g_u}\leq \normW{g}=\lim_{i\to\infty}\normW{\tilde{g}_i}
    \leq I,
  \end{align*}
  as $\{g_i\}_{i=1}^\infty$ is assumed to be a minimizing sequence
  (and $\tilde{g}_i$,
given by Lemma~\ref{lemma:bounded.seq.upper.gradient},
are convex combinations of $g_i$, $g_{i+1}, \ldots$).

  Finally, let us prove that if $u_1$ and $u_2$ are two minimizers,
  then $g_{u_1} = g_{u_2}$. Since $\K_0$ is a convex set, the
  element $u=\tfrac{1}{2}(u_1+u_2)$ is in $\K_f$, and it has a (not
  necessarily minimal) gradient $\tfrac{1}{2}(g_{u_1}+g_{u_2})$. As
  $u\in\K_f$ we must have $\normW{\tfrac{1}{2}(g_{u_1}+g_{u_2})}\geq
  I$. But from the triangle inequality one obtains
  \begin{align*}
    I\leq\normW{\tfrac{1}{2}(g_{u_1}+g_{u_2})}\leq \tfrac{1}{2}\normW{g_{u_1}}
    +\tfrac{1}{2}\normW{g_{u_2}}=I,
  \end{align*}
  which implies that
  $\normW{\tfrac{1}{2}(g_{u_1}+g_{u_2})}=\normW{g_{u_1}}=\normW{g_{u_2}}$. Since
  $W$ is strictly convex, it follows that $g_{u_1}=g_{u_2}$.
\end{proof}

\noindent There is one important special case, where the uniqueness of
solutions can be assured. Namely, when the assignment $u\mapsto g_u$
is linear and $\K_0$ is a linear subspace (and not only a convex set). 

\begin{proposition}\label{prop:u.gu.linear.unique.solution}
  If the map $u\mapsto g_u$ is
  linear and $\K_0\subseteq\SobU$ 
is a nonempty closed linear subspace, which is
  also a Poincar\'e set, 
  then the Dirichlet problem with respect to $\K_f$ has a unique
  solution.
\end{proposition}

\begin{proof}
  By Theorem~\ref{thm:dirichlet.problem}, the Dirichlet
  problem with respect to $\K_f$ has at least one solution.
  Let $u_1$ and $u_2$ be two solutions.
  By Theorem~\ref{thm:dirichlet.problem}
  we know that $g_{u_1}=g_{u_2}$. By assumption, the map $u\mapsto g_u$ is
  linear, from which it follows that $g_{u_1-u_2} = 0$. As
  $\K_0$ is a linear subspace, we have
\begin{equation}   \label{eq-u1-u2-f-K0}
u_1-u_2=(u_1-f)-(u_2-f)\in\K_0.
\end{equation}
 Now, since $\K_0$ is a
  Poincar\'e set, it follows that
  \begin{align*}
    \normV{u_1-u_2}\leq C\normW{g_{u_1-u_2}} = 0,
  \end{align*}
  which implies that $u_1=u_2$.
\end{proof}

\begin{rem}\label{rem:poincare.sets}
In the classical situation, $\K_0$ is usually the zero Sobolev space, and thus a 
closed linear subspace. Here (but for Proposition~\ref{prop:u.gu.linear.unique.solution}) we have merely assumed $\K_0$ to be a closed convex set (with no extra cost, 
as this is enough for the proofs). 
An advantage of this approach is that we can formulate
the obstacle problem in the next section as a Dirichlet problem.
In the obstacle problem the class of competing functions is only convex also in the
classical situation, which traditionally means that it has to be handled separately.
Alternatively one can treat the Dirichlet problem as a special case of the obstacle 
problem, but the opposite direction 
(to treat the obstacle problem as a Dirichlet problem)
is not possible with traditional formulations of the problems.
\end{rem}

\begin{example} \label{ex-grad-max-1}
Let $V=\C$, $W=\R$ and define 
\begin{equation}   \label{eq-grad-max-Re-Im}
   R=\{(u,g) \in V \times W: g \ge \max(|{\Re u}|,|{\Im u}|)\}
   \quad \text{and} \quad
   \K_0=\{u \in V: \Re u \ge 0\}.
\end{equation}
Then $\U=(V,V,W,W,R)$ is a gradient space 
(note that we need an inequality, rather than equality, when
defining $R$, for property \ref{g:sum} to hold).
Moreover, $\K_0$ is a closed convex Poincar\'e set, and $1+ai$ is a solution
of the Dirichlet problem with respect to $\K_f$, with $f=1$,
for all $|a| \le 1$. Thus
we do not always have uniqueness in 
Theorem~\ref{thm:dirichlet.problem}.
(In Examples~\ref{ex-grad-max-1} and~\ref{ex-grad-max-2},
 $i$ denotes the imaginary unit.)
\end{example}

\begin{example} \label{ex-grad-max-2} 
Let $V=W^{1,2}(\R^2,\C)$ (the
  space of complex-valued $W^{1,2}$ functions on $\R^2$, cf.\
  Section~\ref{sect-unweighted-Rn}), $W=L^2(\R^2,\R)$ and define
\[
   R=\{(u,g) \in V \times W: g \ge \max(|\nabla\Re u|,|\nabla \Im u|) 
   \text{ a.e.\  in } \R^2 \},
\]
where $\nabla v$ is the distributional gradient of $v$.
It is easy to verify that  $\U=(V,V,W,W,R)$ is a gradient space 
(again we need an inequality when defining $R$),
and that $\Sob(\U)=W^{1,2}(\R^2,\C)$ 
(with an unorthodox gradient structure similar to the one in 
\eqref{eq-grad-max-Re-Im} and \eqref{eq-grad-struct-alt}). 

Let $\Om=\{x \in \R^2 : 1 < |x| < 2\}$, $\K_0=W^{1,2}_0(\Om,\C)$ 
 be the usual zero Sobolev space (which is a closed convex Poincar\'e set), 
and $f \in \Sob(\U)$ be such that $f(x)=1$ for $|x| \le 1$ and $f(x)=0$ for
$|x| \ge 2$.
By Theorem~\ref{thm:dirichlet.problem}, the Dirichlet problem with respect to
$\K_{f}$ has a solution $v\in\Sob(\U)$. 
Then $u_0:=\Re v$ is a real-valued solution of the same Dirichlet problem,
since every gradient of $v$ is also a gradient of $u_0$.
In fact, $u_0(x)=1-\log_2 |x|$ 
is the unique solution of the classical Dirichlet problem
for harmonic functions with the boundary data $f$.
Let $\phi:[0,1] \to [0,1]$ be a
continuously differentiable function 
such that $\phi(0)=\phi(1)=0$ and $|\phi'(t)| \le 1$
for $0 \le t \le 1$.
Also let $u(x)=u_0(x)+i\phi(u_0(x))$.
Then the minimal gradient of $u$ is
\[
   g_u(x)=\max(|\nabla u_0(x)|,|\phi'(u_0(x))|\,|\nabla u_0(x)|)=|\nabla u_0(x)|,
\]
and $u$ is also a solution of the Dirichlet problem with respect to $\K_f$,
which hence 
does not have a unique solution, even though the existence
of a solution is guaranteed by 
Theorem~\ref{thm:dirichlet.problem}.
\end{example}

\section{The obstacle problem}
\label{sec-obstacle-problem}

\noindent Let us now approach the obstacle problem for gradient spaces. That is,
we would like to solve the Dirichlet problem given the extra
constraint that the solution has to be larger than a given function (the
obstacle). So far, there is no concept of ordering in a
gradient space; hence, in the following we must assume that one may
compare elements of $V$. For this purpose let us recall the definition
of a linear preorder on a vector space.

\begin{definition} \label{deff-preorder}
  Let $\Vt$ be a vector space. A \emph{linear preorder} on $\Vt$ is a binary relation
  $\leq$ such that for $a,b,c\in \Vt$ it holds that
  \begin{enumerate}
  \item $a\leq a$,
  \item if $a\leq b$ and $b\leq c$ then $a\leq c$,
  \item  \label{item-linear-preorder}
    if $b\leq c$ then $a+b\leq a+c$,
  \item if $a\leq b$ and $\alpha\in[0,\infty)$ then $\alpha a\leq \alpha b$.
  \end{enumerate}
  We write $a<b$ if $a\leq b$ and $a\neq b$. The positive cone of $\Vt$,
  i.e.\ elements $a\in \Vt$ such that $a\geq 0$, will be denoted by
  $\Vt_+$.
\end{definition}

\noindent It follows that 
\begin{equation} \label{eq-b<=-a}
   a \le b   \equivalent -b = a-a-b \le b - a-b = -a.
\end{equation}
Note also that we do not assume antisymmetry, i.e.\
$a \le b \le a$ does not necessarily imply that $a=b$, cf.\
however Definition~\ref{def:ordered-grad-space}.

When introducing a linear preorder in a gradient space, it is natural
to assume that it is compatible with the notion of convergence.

\begin{definition}\label{def:preordered.upper.gradient.space}
  A \emph{preordered gradient space} is a gradient space
  $\U=(V,\Vt,W,\Wt,R)$ together with a linear preorder $\leq$ on $\Vt$ such that if
$V \ni u_i\leq\psi \in \Vt$ 
for $i=1,2,\ldots$, and $u_i\to u$ (in $V$), then $u\leq\psi$.
\end{definition}

\noindent
\emph{Assume for the rest of this section that $\U=(V,\Vt,W,\Wt,R)$
is a preordered gradient space.} 

\medskip

\noindent To formulate the obstacle problem for preordered gradient
spaces, we proceed in analogy with the Dirichlet problem: We
choose a subset $\K_0\subseteq\SobU$, an element $f\in\SobU$, an
obstacle $\psi\in \Vt$, and set
\begin{align*}
    \Kpsif = \{u\in\SobU: u-f\in\K_0\textrm{ and }u\geq\psi\}.  
\end{align*}
A solution to the obstacle problem with respect to $\K_{\psi,f}$ is
then given by an element $u\in\K_{\psi,f}$ such that
\[
  \normW{g_u} = \inf_{v\in\Kpsif}\normW{g_v},
\]
where $g_u$ and $g_v$ denote the minimal gradients of $u$ and
$v$, respectively.

To prove the existence of a solution to the obstacle problem, we will
reformulate the obstacle problem as a Dirichlet problem and use
Theorem~\ref{thm:dirichlet.problem} to conclude that a minimizer
exists. For this reason, we need the following result.

\begin{lemma}\label{lemma:tilde.Omega0}
  Let 
  $\K_0\subseteq\SobU$ be a closed convex Poincar\'e set. For
  $f\in\SobU$ and $\psi\in  \Vt$ the set
  \begin{align*}
    \Kt_0(\psi,f) = \{v\in\K_0:v+f\geq\psi\}
  \end{align*}
  is a closed convex Poincar\'e set.
\end{lemma}

\begin{proof}
  Since $\K_0$ is assumed to be closed, it follows from 
Definition~\ref{def:preordered.upper.gradient.space}
that $\Kt_0(\psi,f)$ is also a
  closed set. Moreover, as $\K_0$ is a convex set, it follows that
  $\Kt_0(\psi,f)$ is also convex (as the linear preorder
  respects sums and multiplication by positive real numbers). 
Finally, as
  $\Kt_0(\psi,f)\subseteq\K_0$, and $\K_0$ is assumed to be a
  Poincar\'e set, also $\Kt_0(\psi,f)$ is a Poincar\'e
  set. 
\end{proof}

\noindent
The existence of solutions to obstacle problems can now be
obtained directly from the Dirichlet problem.

\begin{theorem}\label{thm:obstacle.problem.solution}
  For any closed convex Poincar\'e set $\K_0\subset\SobU$, 
$f\in\SobU$ and
  $\psi\in \Vt$, the obstacle problem with respect to $\K_{\psi,f}$ 
  has at least one solution provided that
$\Kpsif\neq\emptyset$. 

If $u_1$ and $u_2$ are solutions, then $g_{u_1}=g_{u_2}$.
Moreover, if $\K_0$ is a linear subspace of $\SobU$ and
the map $u\mapsto g_u$ is linear, then $u_1=u_2$.
\end{theorem}

\begin{proof}
  In the notation of Lemma~\ref{lemma:tilde.Omega0}, the set
  $\Kpsif$ can be described as
  \begin{align*}
    \Kpsif = \{u\in\SobU: u-f\in\Kt_0(\psi,f)\} = \Kt_0(\psi,f) + f,
  \end{align*}
  and since, by the same lemma, the set $\Kt_0(\psi,f)$ is
  a closed convex Poincar\'e set, one can apply 
  Theorem~\ref{thm:dirichlet.problem} 
  to conclude that there exists at least
  one solution as long as $\Kpsif\neq\emptyset$, and
  that the minimal gradient of a solution is unique.

To prove the last part of the statement, we cannot use
Proposition~\ref{prop:u.gu.linear.unique.solution} directly, because
the Poincar\'e set $\Kt_0(\psi,f)$ is not a linear subspace.
However, since it follows from the definition of $\Kt_0(\psi,f)$
that \eqref{eq-u1-u2-f-K0} holds whenever $u_1, u_2 \in\Kpsif$,
the proof of Proposition~\ref{prop:u.gu.linear.unique.solution}
applies also in this case.
\end{proof}

\begin{rem}
It is easily verified that Lemma~\ref{lemma:tilde.Omega0} and 
Theorem~\ref{thm:obstacle.problem.solution} hold also for the
following multi-obstacle problem:
Given $f\in\SobU$ and (possibly uncountable) 
sets $\Psi,\Phi\subset \Vt$, let
\begin{align*}
    \K_{\Psi,\Phi,f} = \{u\in\SobU: u-f\in\K_0\textrm{ and } \psi\le u\leq\phi
\text{ for all } \psi\in\Psi \text{ and } \phi\in\Phi \},
\end{align*}
and find $u\in\K_{\Psi,\Phi,f}$ which minimizes $\normW{g_u}$ among all 
$u\in\K_{\Psi,\Phi,f}$.
Note that it is not always possible to replace $\Psi$ by
$\max_{\psi\in\Psi}\psi$ (or $\Phi$ by $\min_{\phi\in\Phi}\phi$), as the latter bounds
need not exist or need not be optimal with respect to the ordering $\le$.
See Section~\ref{sect-order} for a further discussion on this topic.

For the multi-obstacle problem we thus obtain the following result.
\end{rem}

\begin{theorem}     \label{thm-multi-obstacle}
 For any closed convex Poincar\'e set $\K_0$, 
sets $\Psi,\Phi\subset \Vt$ and $f\in\SobU$, the minimization problem
\[
  \normW{g_u} = \inf_{v\in\K_{\Psi,\Phi,f}}\normW{g_v}
\]
has at least one solution provided that $\K_{\Psi,\Phi,f} \ne \emptyset$. 

If $u_1$ and $u_2$ are solutions, then $g_{u_1}=g_{u_2}$.
Moreover, if $\K_0$ is a linear subspace of $\SobU$ and
the map $u\mapsto g_u$ is linear, then $u_1=u_2$.
\end{theorem}

\noindent The assumption that $\Kpsif$ is nonempty in 
Theorem~\ref{thm:obstacle.problem.solution} can be rephrased in the
following way.

\begin{proposition}\label{prop:K.nonemtpy.criterion}
  Let $\K_0\subseteq\SobU$, $f\in\SobU$ and $\psi\in \Vt$. Then
  $\Kpsif\neq\emptyset$ if and only if there exists $v\in\K_0$ such
  that $v \ge \psi-f$.
\end{proposition}

\begin{proof}
  First, assume that there is $u\in\Kpsif$. By
  setting $v=u-f$ it follows that $v\in\K_0$ and $v=u-f\geq\psi-f$
  since $u\geq\psi$. 

Conversely,
assume that there exists 
  $v\in\K_0$ such that $v\geq\psi-f$. Defining $u=v+f$ it follows
  that $u\in\SobU$ (since $\K_0\subseteq\SobU$ and $f\in\SobU$,
  using \ref{g:sum}), $u-f=   v\in\K_0$ and $u=v+f\geq
  \psi-f+f=\psi$. Hence, $u\in\Kpsif$.
\end{proof}

\begin{rem}
The space $\Vt$ in a gradient space $\U=\paraa{V,\Vt,W,\Wt,R}$
played a prominent role of a preordered space in this section, 
and will also play a vital role in Section~\ref{sect-order}.
It however, does not  play
any role in Sections~\ref{subsect-DP} and~\ref{sec:rayleigh.quotient}.
On the other hand, the space $\Wt$ does not play any direct role in this paper,
but we have chosen to still include it for symmetry reasons and
possible later applications.
\end{rem}

\section{Minimizing the Rayleigh quotient}
\label{sec:rayleigh.quotient}

\noindent
\emph{Assume in  this section that $\U=(V,\Vt,W,\Wt,R)$
is a gradient space.}

\medskip

\noindent In this section we consider another variational problem
that can be treated using gradient spaces.  Weak eigenfunctions of the
Laplace operator can be found by minimizing the \emph{Rayleigh
  quotient} 
\begin{align*}
  \Ray(u)=\frac{\norm{\nabla u}_{L^2}}{\norm{u}_{L^2}},
\end{align*}
whose infimum gives the first (apart from 0) eigenvalue of the Laplace
operator. If there exists a Poincar\'e inequality on the set over
which the infimum is computed, i.e.\ if
\begin{align*}
  \norm{u}_{L^2}\leq C\norm{\nabla u}_{L^2}
\end{align*}
for some $C>0$, then it is clear that a positive infimum
exists. To guarantee the existence of a minimizer, one can use the
fact that the Sobolev space $W^{1,2}$ is compactly embedded into $L^2$ to
find a sequence converging to a minimizer. Therefore, in the setting of
gradient spaces, we shall minimize over sets which satisfy a version
of the Rellich--Kondrachov theorem.

\begin{definition}
  Let  $\K\subseteq\SobU$ be a
    cone, i.e.\ if $u\in\K$ then $\alpha u\in\K$ for all
    $\alpha>0$. The cone $\K$ is a \emph{Rellich--Kondrachov cone} if for
  every bounded sequence $\{u_i\}_{i=1}^\infty\subseteq\K$, such that
  $\{g_{u_i}\}_{i=1}^\infty$ is also a bounded sequence, there exists
  a convergent subsequence $\{u_{i_k}\}_{k=1}^\infty$ which converges
(in $V$) 
  to an element $u\in \K$. Moreover, a Rellich--Kondrachov cone is
  \emph{regular} if, for $u\in\K$, $g_u=0$ implies that $u=0$.
\end{definition}

\noindent 
It follows directly from the definition that Rellich--Kondrachov cones
are closed subsets of $\SobU$.

Note that the regularity assumption is necessary for any set to be a
Poincar\'e set. Furthermore, it is sufficient to guarantee that a
Rellich--Kondrachov cone supports a Poincar\'e inequality.

\begin{proposition}\label{prop:rk.set.poincare}
  A regular Rellich--Kondrachov cone is a Poincar\'e set.
\end{proposition}

\begin{proof}
  Let us assume that the regular Rellich--Kondrachov cone $\K$ is not
  a Poincar\'e set, i.e.\ there is no $C>0$ such that
  \[
    \normV{u}\leq C\normW{g_u}
  \]
  for all $u\in\K$. We will now show that, under this assumption, one
  may construct an element $u\in\K$ with $u\neq 0$ and $g_u=0$, which
  contradicts the fact that $\K$ is regular. If there is no Poincar\'e
  inequality, then one may find a sequence $\{\tilde{u}_k\}_{k=1}^\infty$ 
  in $\K$
such that 
  \begin{align*}
    \normV{\tilde{u}_k}\geq k\normW{g_{\tilde{u}_k}}.
  \end{align*}
  Since (for $\alpha>0$) it holds that
  \begin{align*}
    \normV{\alpha\tilde{u}_k}\geq k\normW{\alpha g_{\tilde{u}_k}}
    = k\normW{g_{\alpha\tilde{u}_k}}
  \end{align*}
  one may construct a rescaled sequence
  $u_k=\tilde{u}_k/\normV{\tilde{u}_k}$ 
with the same property. Thus,
  one obtains a sequence $\{u_k\}_{k=1}^\infty$ in $\K$ (as $\K$ is a cone)
with $\normV{u_k}=1$ and
  $\normV{u_k}\geq k\normW{g_{u_k}}$, which implies that
  \begin{align*}
    \normW{g_{u_k}}\leq\frac{1}{k}.
  \end{align*}
  Thus, $g_{u_k}\to 0$ which, in particular, implies that
  $\{g_{u_k}\}_{k=1}^\infty$ is bounded. Now, since both sequences 
  $\{u_k\}_{k=1}^\infty$ and
  $\{g_{u_k}\}_{k=1}^\infty$ are bounded, and $\K$ is a
  Rellich--Kondrachov cone, we conclude that there exists a convergent
  subsequence $\{u_{i_k}\}_{k=1}^\infty$, converging to some $u\in\K$. Clearly,
  $\normV{u}=1$ which implies that $u\neq 0$. Moreover, 
  $g_{u_{i_k}} \to 0$, and from
  \ref{gs:upper.gradient.limit} it follows that $(u,0)\in R$, which
  contradicts the fact that $\K$ is assumed to be regular. Hence,
  $\K$ is a Poincar\'e set. 
\end{proof}

\begin{theorem}  \label{thm-solve-rayleigh}
  Let $\K\subseteq\SobU$ be a
  nonempty regular Rellich--Kondrachov cone. Then there exists an
  element $u\in\K$ such that
  \begin{align*}
    \frac{\normW{g_u}}{\normV{u}} = \inf_{\substack{v\in\K\\v\neq 0}}\frac{\normW{g_v}}{\normV{v}}>0.
  \end{align*}
\end{theorem}

\begin{proof}
  For any $u\in\SobU$, such that $u\neq 0$, let
  \begin{align*}
    \Ray(u) = \frac{\normW{g_u}}{\normV{u}}.
  \end{align*}
  For $\alpha>0$ it follows that
  \begin{align}\label{eq:rayleigh.scaling}
    \Ray(\alpha u)=\frac{\normW{g_{\alpha u}}}{\normV{\alpha u}}
    =\frac{\normW{\alpha g_u}}{\normV{\alpha u}}=\frac{\normW{g_u}}{\normV{u}}=\Ray(u).
  \end{align}
  Since $\K$ is a Poincar\'e set (by 
  Proposition~\ref{prop:rk.set.poincare}), 
  there exists $C>0$ such that $\Ray(v)\geq
  C$ for all $v\in\K$, which implies that the infimum
  \begin{align*}
    I = \inf_{\substack{v\in\K\\v\neq 0}}\Ray(v)
  \end{align*}
  is positive. Let $\{\tilde{u}_i\}_{i=1}^\infty$ be a minimizing sequence, i.e.\ a
  sequence such that
  \begin{align*}
    I = \lim_{i\to\infty}\Ray(\tilde{u}_i).
  \end{align*}
  The normalized sequence given by
  $u_i=\tilde{u}_i/\normV{\tilde{u}_i}$ 
clearly fulfills $\Ray(u_i)\geq
  I$ (by the definition of $I$ as the infimum), and from
  \eqref{eq:rayleigh.scaling} one obtains
  \begin{align*}
    I\leq \Ray(u_i)= \Ray(\tilde{u}_i),
  \end{align*}
  which implies that $\lim_{i\to\infty} \Ray(u_i)=I$, i.e.\ $\{u_i\}_{i=1}^\infty$ is 
also a
  minimizing sequence. Moreover, it follows that $\{g_{u_i}\}_{i=1}^\infty$ is a
  bounded sequence since 
  \begin{align*}
    I=\lim_{i\to\infty}\Ray(u_i)=\lim_{i\to\infty}\normW{g_{u_i}}.
  \end{align*}
  As $\K$ is a Rellich--Kondrachov cone one can find
  a convergent subsequence $\{u_{i_k}\}_{k=1}^\infty$, converging to some
  $u\in\K$, which must have $\normV{u}=1$. From the two bounded
  sequences $\{u_{i_k}\}_{k=1}^\infty$ and $\{g_{u_{i_k}}\}_{k=1}^\infty$ 
  one can use
  Lemma~\ref{lemma:bounded.seq.upper.gradient} to construct
  convex combinations $\{\hat{u}_i\}_{i=1}^\infty$ and 
  $\{\hat{g}_i\}_{i=1}^\infty$ which
  converge to $u$ and $g$, respectively, with $(u,g)\in R$. Since 
$\hat{g}_k$,
given by Lemma~\ref{lemma:bounded.seq.upper.gradient},
are convex combinations of 
$g_{u_{i_k}}$, $g_{u_{i_{k+1}}}, \ldots$,
we get 
  \begin{align*}
    \normW{g} =\lim_{i \to \infty}\normW{\hat{g}_i} 
    \leq \limsup_{k\to\infty} \|g_{u_{i_k}}\|_W = I.
  \end{align*}
  It follows  that
  \begin{align*}
    I\leq \Ray(u)=\frac{\normW{g_u}}{\normV{u}} = \normW{g_{u}}\leq
    \normW{g}\leq I
  \end{align*}
  which shows that $\Ray(u)=I$.
\end{proof}

\section{Least upper bounds and lattice structures}
\label{sect-order}

\noindent For two real-valued functions,
one may define their least upper bound 
by a pointwise choice of the
maximum of the two function values. However, a pointwise construction
is not available in the setting of (preordered) gradient spaces, and we shall
investigate to what extent one may define an upper bound with the
help of a minimization problem.
The idea behind this approach is that the pointwise maximum of
two nonnegative functions $f$ and $g$ minimizes the $L^p$-norm among all
functions $h$ such that $h\ge f$ and $h\ge g$.

More generally, for two positive elements $\psi_1,\psi_2\in \Vt$, 
the aim will be to find an
element $u\in V$ such that $u\geq\psi_1$, $u\geq\psi_2$ and 
\begin{align*}
  \normV{u} = \inf_{v}\normV{v},
\end{align*}
where the infimum is taken over all $v\in V$ such that $v\geq\psi_1$
and $v\geq\psi_2$. Such an upper bound will not necessarily be a least
upper bound, with respect to the linear ordering, but is only an upper
bound with minimal norm. In fact, requiring  that the above
construction yields a least upper bound in general leads to severe
restrictions on the underlying spaces in the case of noncommutative
algebras.

To reach the desired results, we have to refine our notion of ordering
in a gradient space. Namely, we shall assume that the norm is
compatible with the ordering in the following sense.

\begin{definition} \label{def:ordered-grad-space}
  An \emph{ordered gradient space} $\U=(V,\Vt,W,\Wt,R)$ is a
  preordered gradient space such that $V$ is a strictly
  convex Banach space and for $u,v\in V$ we have
  \begin{equation}  \label{eq-ordered}
    0\leq u \le v\implies \normV{u} \le \normV{v}.
  \end{equation}
\end{definition}

\noindent
\emph{Assume for the rest of this section that $\U=(V,\Vt,W,\Wt,R)$
is an ordered gradient space.}

\medskip

\noindent 
It follows that the preorder $\le$  will be a partial
order on $V$, i.e.\ if $u,v \in V$,  $u\leq v$ and $v\leq u$, then $u=v$.
Indeed, if $u,v \in V$,  $u\leq v$ and $v\leq u$, then  using 
\eqref{item-linear-preorder} in Definition~\ref{deff-preorder}
shows that
\[
    0 = u-u \le v-u \le u-u =0,
\]
and hence by \eqref{eq-ordered}, $\normV{v-u} \le \normV{0} =0$,
so $u=v$.

\begin{example} \label{ex-CR}
Note, however, that it does not follow that $\le$ is a partial order on $\Vt$.
To see this, let $V=C(\R)$ with sup-norm, 
\[
    \Vt=\{f : f=h \text{ a.e.\ for some } h \in V\}
\]
and let $u \le v$ if $u \le v$ pointwise a.e.\ as functions
(where a.e.\ refers to the Lebesgue measure).  Then
$\le$ is a partial order on $V$ but only a preorder on $\Vt$.  One can
clearly choose 
 $W$, $\Wt$ and $R$ (in many ways) so as to make it
into an ordered gradient space.
\end{example}

In fact, for the results in this section, the spaces $W$ and $\Wt$, as
well as the gradient relation $R$, will not play any role, it is only
$V$ and $\Vt$ and the conditions imposed on them that will be involved
here.  However, the discussion is still important to better
  understand the
concept of ordered gradient spaces.

The following observation will be of use to us.

\begin{lemma} \label{lem-strict}
Let $u,v \in V$. 
Then 
  \[
    0\leq u<v\implies \normV{u}<\normV{v}.
  \]
\end{lemma}

\begin{proof}
Assume that $0 \le u <v$.
By \eqref{eq-ordered} we know that $\normV{u} \le \normV{v}$.
Assume that $\normV{u}=\normV{v}$ and let $w=\frac{1}{2}(u+v)$.
Then $u < w < v$ and thus, by \eqref{eq-ordered} again,
$
  \normV{u}\le \normV{w}  \le \normV{v}$.
Hence $\normV{u} = \normV{v}= \normV{w} =\normV{\tfrac{1}{2}(u+v)}$,
and since $V$ is strictly convex  $u=v$, which contradicts
that $u < v$.
Therefore $\normV{u} <\normV{v}$.
\end{proof}

\noindent We are now ready to solve the minimization problem which
constructs an upper bound 
of two elements in $\Vt$ with least possible norm.

\begin{proposition}\label{prop:existence.maximum}
For $\psi_1,\psi_2\in \Vt$ set 
\[
\Om= \Omega(\psi_1,\psi_2)=\{u\in V:
  u\geq\psi_1\textrm{ and }u\geq\psi_2\}.
\]
If $\Omega\neq\emptyset$
  then there exists a unique element $u\in V$ such that
  \begin{align*}
    \normV{u} = \inf_{v\in\Omega}\normV{v}.
  \end{align*}
\end{proposition}

\begin{proof}
  Let $\{u_i\}_{i=1}^\infty$ be a minimizing sequence in $\Omega$ with
  \begin{align*}
    \lim_{i\to\infty}\normV{u_i} = \inf_{v\in\Omega}\normV{v} =: I.
  \end{align*}
  Since $\{u_i\}_{i=1}^\infty$ is bounded there exists, by
  Corollary~\ref{cor:weak.convex.combination}, a 
sequence
  $\tilde{u}_i= \sum_{j=i}^{N_i} \alp_{ij} u_j$ 
   converging to some $u$.
It follows directly that 
 $\tilde{u}_i\geq\psi_1$ and $\tilde{u}_i\geq\psi_2$ for
  $i=1,2,\ldots$. From the definition of a preordered gradient space,
  it follows that $u\geq\psi_1$ and $u\geq\psi_2$, which implies that
  $u\in\Omega$. Furthermore, 
  since $\ut_i$ is a convex combination of $\{u_j\}_{j=i}^\infty$,
  it is clear that
  \begin{align*}
    I\leq\normV{u} = \lim_{i\to\infty}\normV{\tilde{u}_i} \leq
    \limsup_{i\to\infty}\normV{u_i} = I,
  \end{align*}
  which proves that there exists a minimizer in $\Omega$. 

  Next, let us
  prove uniqueness. Assume that $u_1$ and $u_2$ are two minimizers.
  In particular,  $\normV{u_1}=\normV{u_2}=I$. 
  Since $\frac{1}{2}(u_1+u_2)\geq\psi_1$ and
  $\frac{1}{2}(u_1+u_2)\geq \psi_2$, 
  we get that $\frac{1}{2}(u_1+u_2)\in\Omega$ and thus
   $I\leq\normV{\frac{1}{2}(u_1+u_2)}$. From the triangle inequality one
  obtains
  \begin{align*}
    I\leq\normV{\tfrac{1}{2}\paraa{u_1+u_2}}\leq\tfrac{1}{2}\normV{u_1}
    +\tfrac{1}{2}\normV{u_2} = I,
  \end{align*}
  which implies that
  $I=\normV{u_1}=\normV{u_2}=\normV{\frac{1}{2}(u_1+u_2)}$. As $V$ is
  assumed to be strictly convex, it follows that $u_1=u_2$.
\end{proof}

\begin{rem} \label{rem-multi-obst}
Proposition~\ref{prop:existence.maximum} and its proof can also be regarded
as a special case of the multi-obstacle problem.
For this, one defines a new gradient relation 
$R=\{(u,u):u \in V\}$ with $\Wt=\Vt$ and $W=V$.
The existence and uniqueness 
then follow directly from Theorem~\ref{thm-multi-obstacle}.
Note that the minimizing sequence is bounded in this case and thus 
Remark~\ref{rmk-bdd-minim-seq-no-PI} applies here.
In any case, $V$ is automatically
a Poincar\'e set with the above choice of $R$.
\end{rem}

\noindent 
Thus, a necessary and sufficient condition for such a minimizer to
exist is that there exists at least one upper bound of $\psi_1$ and
$\psi_2$ in $ V$. Therefore, we introduce the following subset of $\Vt$. 

\begin{definition}
The
  \emph{$V$-bounded positive cone} is given as
  \begin{align*}
    \Bplus = \{\psi\in \Vt: 0 \le \psi \le u \text{ for some }u\in V\}.
  \end{align*}
We also define $\Vplus=V \cap \Bplus= \{v\in V: v\geq 0\}$.
\end{definition}

\noindent 
Note that $\Vplus$ is closed, which follows immediately from
  the assumption that $\U$ is a preordered gradient space
  (cf.\ Definition~\ref{def:preordered.upper.gradient.space}).

If $\psi_1,\psi_2\in\Bplus$ then there exist $u_1,u_2\in V$ such
that $\psi_i\leq u_i$ for $i=1,2$. In particular, it follows that
$u_1+u_2\geq\psi_1+\psi_2\geq\psi_i$ for $i=1,2$.
Hence, there exists a unique solution to the minimization problem in
Proposition~\ref{prop:existence.maximum} for any $\psi_1,\psi_2\in\Bplus$.

Consequently, we can introduce the maximum of two $V$-bounded elements.

\begin{definition}\label{def:maximum}
Let $\psi_1,\psi_2\in\Bplus$.  The unique minimizer in 
  Proposition~\ref{prop:existence.maximum} 
  is called the \emph{maximum of $\psi_1$
    and $\psi_2$} and is denoted 
$\max(\psi_1,\psi_2)$.
\end{definition}

\noindent Note that 
$\max(\psi_1,\psi_2)\in \Vplus$ since $\max(\psi_1,\psi_2)\in V$ and
$\max(\psi_1,\psi_2)\geq\psi_1\geq 0$.
(We do not define $\max(\psi_1,\psi_2)$ unless $\psi_1,\psi_2 \in
\Bplus$.)
However, $\max(\psi_1,\psi_2)$ is not necessarily a least
upper bound with respect to the ordering on $V$ 
(see the example in Section~\ref{sect-matrix-alg}).
Nevertheless, it enjoys the following properties.

\begin{lemma}\label{lemma:max.properties}
Let $\psi_1,\psi_2\in\Bplus$ and $u\in V$. 
Then the following are true\/\textup{:}
 \begin{enumerate}
  \renewcommand{\theenumi}{\textup{(\arabic{enumi})}}%
  \renewcommand{\labelenumi}{\theenumi}%
  \item if $u\geq \psi_1$, then $\max(u,\psi_1)=u$,\label{max.prop.comparable}
  \item if $u\geq\psi_1$, $u\geq\psi_2$ and
    $u\neq\max(\psi_1,\psi_2)$, then
    $\normV{u}>\normV{\max(\psi_1,\psi_2)}$.\label{max.prop.norm.ineq}
  \end{enumerate}
\end{lemma}

\begin{proof}
  \ref{max.prop.comparable} Assume that $u\geq\psi_1$.
Then $u \in \Om(u, \psi_1)$. 
By \eqref{eq-ordered}, $\normV{v}\ge\normV{u}$ for $v \in  \Om(u, \psi_1)$,
and hence
$u$ has
  minimal norm in $ \Om(u, \psi_1)$. Thus, $\max(u,\psi_1)=u$.

  \ref{max.prop.norm.ineq} Assume that $u\geq\psi_1$ and
  $u\geq\psi_2$. As $u\in\Omega(\psi_1,\psi_2)$, 
and $\normV{\max(\psi_1,\psi_2)}$
  is the infimum of $\normV{v}$ among
  $v\in\Omega(\psi_1,\psi_2)$, 
we must have 
  $\normV{u}\geq\normV{\max(\psi_1,\psi_2)}$. Moreover, since the
  minimizer is unique, we must have
  $\normV{u}>\normV{\max(\psi_1,\psi_2)}$ whenever
  $u\neq\max(\psi_1,\psi_2)$.
\end{proof}

\begin{rem}
Both for the proof of Proposition~\ref{prop:existence.maximum} 
and the one outlined in Remark~\ref{rem-multi-obst},
it is enough to require that $\U$ is a preordered gradient space with the
additional requirement that $V$ be strictly convex.
However, for $u \in \Vplus$ it is natural to require that
$\max(u,u)=u$, 
and this is true if and only if $\U$ is an ordered gradient space,
which can be seen as follows:
If there are $u$ and $v$ such that $0 \le u \le v$ but 
$\normV{u} > \normV{v}$, then $u$ cannot be the solution of the
minimization problem for $\max(u,u)$, and thus $\max(u,u) \ne u$.
The converse direction follows from 
Lemma~\ref{lemma:max.properties}\,\ref{max.prop.comparable}.
\end{rem}

\noindent The least upper bound property of $\max(\psi_1,\psi_2)$ is
problematic, since an upper bound $v$ of $\psi_1$ and $\psi_2$ need not be
comparable to $\max(\psi_1,\psi_2)$ 
(see the example in Section~\ref{sect-matrix-alg}).
Here we need to be a bit more precise. If $\psi_1, \psi_2 \in \Bplus$,
then they may have many upper bounds in $\Bplus \setm \Vplus$.
For our purposes we will however (mainly)
be interested in upper bounds that
belong to $\Vplus$, and their minimality within this class. 
In order to clarify this, we make the following definition.

\begin{definition}
An upper bound $v \in \Vplus$ of $\psi_1,\psi_2 \in \Bplus$ 
is a \emph{$\Vplus$-least upper bound}  of $\psi_1$ and $\psi_2$
if $v \le u$ for any
upper bound $u \in \Vplus$  of $\psi_1$ and $\psi_2$.

\emph{$\Vplus$-greatest lower bounds} are defined analogously.
\end{definition}

Since $\le$ is a partial order on $V$ it follows that
if a $\Vplus$-least upper bound of $\psi_1$ and $\psi_2$ exists
then it is unique. 
Similar uniqueness holds
for $\Vplus$-greatest lower bounds.

\begin{proposition}\label{prop:lub.comparable}
  If $v\in V$ is such that $v\geq\psi_1\in\Bplus$, $v\geq\psi_2\in\Bplus$ and $v$ is
  comparable to $\max(\psi_1,\psi_2)$, then $v\geq\max(\psi_1,\psi_2)$.
In particular, if $\psi_1$ and $\psi_2$ have a $\Vplus$-least 
upper bound $v$, then $v=\max(\psi_1,\psi_2)$.
\end{proposition}

\begin{proof}
  If $v$ is comparable to $\max(\psi_1,\psi_2)$ then, by definition,
  either $v\geq\max(\psi_1,\psi_2)$ or $v<\max(\psi_1,\psi_2)$. Since
  $v\in\Omega(\psi_1,\psi_2)$, the statement $v<\max(\psi_1,\psi_2)$, which implies that
  $\normV{v}<\normV{\max(\psi_1,\psi_2)}$ (by Lemma~\ref{lem-strict}),
  contradicts the fact that
  $\max(\psi_1,\psi_2)$ minimizes the norm in $\Omega(\psi_1,\psi_2)$. Hence, 
  $v\geq\max(\psi_1,\psi_2)$.  

If, in addition, $v$ is a
  $\Vplus$-least upper bound of $\psi_1$ and $\psi_2$, then one cannot have
  $v>\max(\psi_1,\psi_2)$ (since that would contradict the assumption
  that $v$ is a \emph{least} upper bound), and it follows that
  $v=\max(\psi_1,\psi_2)$.
\end{proof}

\noindent Let us now introduce another variational problem in order to
define the minimum of two elements $\psi_1,\psi_2\in\Bplus$. The idea is to
minimize the (norm)-distance to elements that are lower bounds
of $\psi_1$ and $\psi_2$. However, as $\psi_1$ and $\psi_2$ are not
necessarily elements of $V$ (and, in particular, $\normV{\psi_1-u}$
need not be defined for $u\in V$) we shall instead minimize the
distance to $\max(\psi_1,\psi_2)\in V$.   

\begin{proposition}\label{prop:min.via.minimization}
For $\psi_1,\psi_2\in\Bplus$
we set
\[
\Omega'=\Om'(\psi_1,\psi_2)=\{v\in V: 0 \le  v\leq\psi_1\textrm{ and
  }v\leq\psi_2\} \subset \Vplus.
\] 
  Then there exists a unique element $u\in\Omega'$ such that
  \begin{align*}
    \normV{\max(\psi_1,\psi_2)-u} = \inf_{v\in\Omega'}\normV{\max(\psi_1,\psi_2)-v}.
  \end{align*}  
\end{proposition}

\begin{proof}
  The set $\Omega'$ contains the element $0$ and is therefore always
  nonempty. Set $M=\max(\psi_1,\psi_2)$ and let
  $\{u_i\}_{i=1}^\infty$ be a minimizing sequence in $\Omega'$, i.e.
  \begin{align*}
    \lim_{i\to\infty}\normV{M-u_i} = \inf_{v\in\Omega'}\normV{M-v}.
  \end{align*}
  As $\{u_i\}_{i=1}^\infty$ is a bounded sequence (due to $0\leq
  u_i\leq\psi_1\in\Bplus$), there exists (by
  Corollary~\ref{cor:weak.convex.combination}) a sequence $\{\tilde{u}_i\}_{i=1}^\infty$
converging to some
  $u$, where each $\tilde{u}_i$ is a convex combination of 
  $\{u_j\}_{j=i}^\infty$.
  It follows from the
  definition of a preordered gradient space that
  $u\in\Omega'$. Moreover, 
  \begin{align*}
    I\leq\normV{M-u}=\lim_{i\to\infty}\normV{M-\ut_i}
    \leq\lim_{i\to\infty}\normV{M-u_i} = I,
  \end{align*}
  which shows that $u$ is indeed a minimizer. 

  Next, let $u_1$ and $u_2$ be
  two minimizers fulfilling
  $\normV{M-u_1}=\normV{M-u_2}=I$. Since $u_1,u_2\in\Omega'$
  it is clear that $\frac{1}{2}(u_1+u_2)\in\Omega'$, and one obtains
  \begin{align*}
    I\leq \normV{M-\tfrac{1}{2}(u_1+u_2)}\leq 
    \tfrac{1}{2}\normV{M-u_1}+\tfrac{1}{2}\normV{M-u_2}=I,
  \end{align*}
  which shows that
  \begin{align*}
    \normV{M-u_1}=\normV{M-u_2}=
    \tfrac{1}{2}\normV{M-u_1+M-u_2}.
  \end{align*}
  As $V$ is assumed to be strictly convex it follows that $u_1=u_2$.
\end{proof}

\noindent Let us now make the following definition.

\begin{definition}
For $\psi_1,\psi_2\in\Bplus$
  we define $\min(\psi_1,\psi_2)$ to be the (unique) minimizer in
  Proposition~\ref{prop:min.via.minimization}.
\end{definition}

\noindent 
Our next result shows that the existence of $\Vplus$-least upper bounds implies
the existence of $\Vplus$-greatest lower bounds and thus
that $\Vplus$ is a lattice.
Recall that a partially ordered set $(A,\leq)$ is a \emph{lattice} if every
pair of elements has a greatest lower bound and a least
upper bound in $A$.

\begin{proposition}\label{prop:Vplus.lattice}
If 
there exists a $\Vplus$-least upper bound\/ 
\textup{(}which then must be unique\/\textup{)}
for all pairs $u_1,u_2\in
  \Vplus$,
then $\Vplus$ is a lattice,
and 
moreover   $\max(u_1,u_2)$ is the $\Vplus$-least upper bound of $u_1$ and $u_2$,
and $\min(u_1,u_2)$ is the $\Vplus$-greatest lower bound of $u_1$ and $u_2$.
\end{proposition}

\begin{proof}
First note that $\max(u_1,u_2)$ is the $\Vplus$-least upper bound of $u_1$ and $u_2$,
by Proposition~\ref{prop:lub.comparable}.
  Let $v\in\Vplus$ be a lower bound of $u_1$ and $u_2$ and set
  \begin{align*}
    \hat{v} = \max(v,\min(u_1,u_2)).
  \end{align*}
  Since $u_1$ and $u_2$ are upper bounds for both $v$ and
  $\min(u_1,u_2)$, we have 
$\hat{v}\leq u_i$ for $i=1,2$
  as $\hat{v}$ is the least such upper bound (by
  Proposition~\ref{prop:lub.comparable} again). 
Hence $\max(u_1,u_2)-\hat{v}\geq 0$.
Using \eqref{eq-b<=-a} and $\hat{v}\geq\min(u_1,u_2)$, we then conclude that
\begin{align*}
    0\leq \max(u_1,u_2)-\hat{v}\leq \max(u_1,u_2)-\min(u_1,u_2),
  \end{align*}
from which it follows that 
  \begin{align*}
    \normV{\max(u_1,u_2)-\hat{v}}\leq\normV{\max(u_1,u_2)-\min(u_1,u_2)},
  \end{align*}
  as $\U$ is an ordered gradient space. Since $\min(u_1,u_2)$ is
  the unique minimizer of the above norm 
(by  Proposition~\ref{prop:min.via.minimization}), we must have
  $\hat{v}=\min(u_1,u_2)$. Hence, 
  \begin{align*}
    \min(u_1,u_2) = \max(v,\min(u_1,u_2)),
  \end{align*}
  which implies that $\min(u_1,u_2)\geq v$.
Thus $u_1$ and $u_2$ have a $\Vplus$-greatest lower bound which equals
$\min(u_1,u_2)$, and hence $\Vplus$ is a lattice.
\end{proof}

\noindent We can now obtain the following characterization.

\begin{theorem} \label{thm-char-lattice}
The following are equivalent\/\textup{:}
  \begin{enumerate}
  \renewcommand{\theenumi}{\textup{(\alph{enumi})}}%
  \renewcommand{\labelenumi}{\theenumi}%
  \item \label{i-lub}
  there is a $\Vplus$-least upper bound   
    for all $u_1, u_2 \in \Vplus$,
  \item \label{i-max}
  $\max(u_1,u_2)$ is the $\Vplus$-least upper bound  
    for all $u_1, u_2 \in \Vplus$,
  \item \label{i-lattice}
    $\Vplus$ is a lattice,
  \item \label{i-glb}
  there is a $\Vplus$-greatest lower bound  
    for all $u_1, u_2 \in \Vplus$,
  \item \label{i-min}
  $\min(u_1,u_2)$ is the $\Vplus$-greatest lower bound 
    for all $u_1, u_2 \in \Vplus$.
  \end{enumerate}
\end{theorem}

\begin{proof}
\ref{i-lub} \imp \ref{i-max}
This follows from Proposition~\ref{prop:lub.comparable}.

\ref{i-max} \imp \ref{i-lattice}
and \ref{i-max} \imp \ref{i-min}
This is the conclusion of Proposition~\ref{prop:Vplus.lattice}.

\ref{i-lattice} \imp \ref{i-lub} 
This follows from the definition of lattice. 

\ref{i-min} \imp \ref{i-glb} This is trivial.

To complete the proof we show that \ref{i-glb} \imp \ref{i-max}.
Let $u_1,u_2 \in \Vplus$ and $u=\max(u_1,u_2) \in \Vplus$.
Also let $\ut \in \Vplus$ be an arbitrary upper bound of $u_1$ and $u_2$.
Let $v$ be  the $\Vplus$-greatest lower bound of $u$ and $\ut$, 
which exists by assumption.
As 
$u_i$
is also a common lower bound, we have $u_i \le v$, $i=1,2$.
Hence $v$ is an upper bound of $u_1$ and $u_2$.
As $u$ is the upper bound with least norm, 
we have $\normV{v} \ge  \normV{u}$.
Since also $v \le u$, Lemma~\ref{lem-strict} implies that $v=u$.
Thus $u=v \le \ut$, showing that 
$u=\max(u_1,u_2)$ is the $\Vplus$-least upper bound of $u_1$ and $u_2$.
As $u_1$ and $u_2$ were arbitrary, we have shown \ref{i-max}.
\end{proof}

\noindent A natural question to ask is whether $V$ (and not only
$\Vplus$) is a lattice in this setting? As the ordering is linear, one
can show that as soon as every pair of elements in $V$ has a
common upper bound (in $V$), it follows that
$V$ is a lattice; clearly this is also a 
necessary condition. 
One can refine this a bit, and in order to do so let us introduce the
linear subspace (see Lemma~\ref{lem-subspace}) 
\[
V_0:=\{u \in V : u \le v \text{ for some } v \in \Vplus\}
=\{u-v:u,v\in \Vplus\}
\]
of $V$. 
We then have the following results
(whose proofs we postpone to the end of this section).

\begin{proposition} \label{prop-V0-lattice}
$V_0$ is a lattice if and only if $\Vplus$ is a lattice.
In this case, for all $u_1,u_2 \in \Vplus$, their
$\Vplus$-least upper bound $\max(u_1,u_2)$
is even least among all upper bounds in $V$.

Similarly, the  $\Vplus$-greatest lower bound $\min(u_1,u_2)$
is greatest also among all lower bounds of $u_1$ and $u_2$ in $V$.
\end{proposition}

\noindent
In particular, it follows that one can replace 
``$\Vplus$-least'' and ``$\Vplus$-greatest''
by ``$V$-least'' and ``$V$-greatest'' (with the obvious interpretation)
in Theorem~\ref{thm-char-lattice} 
(while leaving the other $\Vplus$'s) to produce four more equivalent statements.

\begin{theorem} \label{thm-V-lattice}
$V$ is a lattice if and only if $V=V_0$ and $\Vplus$ is a lattice.
\end{theorem}

\noindent
Note that $V=V_0$ if and only if every element in $V$ has an upper bound in $\Vplus$,
or equivalently, if every pair of elements in $V$ has a common upper bound in $V$.

The condition $V=V_0$ cannot be dropped, as seen by the following
example.

\begin{example} Let  $V=\Vt=\C$ and say that
$u \le v$ if $\Re u \le \Re v$ and $\Im u = \Im v$.
Then $\Vplus=\{\lambda\in\R: \lambda\geq 0\}$ and $V_0=\R \ne V$.
Here  $\Vplus$ and $V_0$ are lattices, but $V$ is not.
\end{example}

\noindent 
We do not know if $V_0$ is always closed, but we next show 
that $V_0$ is indeed a linear subspace.

\begin{lemma} \label{lem-subspace}
$V_0$ is a linear subspace.
\end{lemma}

\begin{proof}
Let $u,v \in V_0$ and $\al \ge 0$. Then
it is rather obvious that $u+v,\al u \in V_0$.
The only nonobvious fact we need to show is that $-u \in V_0$.
As $u \in V_0$, there is $w \in \Vplus$ such that $u \le w$.
Then $0 \le w-u \in V$ and $w \ge 0$. Hence
$
  -u \le w-u \in \Vplus.
$
\end{proof}

\begin{proof}[Proof of Proposition~\ref{prop-V0-lattice}]
First assume that $\Vplus$ is a lattice,
and let $u_1,u_2 \in V_0$.
We need to show that $u_1$ and $u_2$ have a $V_0$-least upper bound.
By Lemma~\ref{lem-subspace},  $-u_i \in V_0$  and thus it  has 
an upper bound $v_i \in \Vplus \subset V_0$, $i=1,2$.
It follows that $w:=-v_1-v_2 \in V_0$ is a common lower bound of $u_1$ and $u_2$.
Hence, $0\le u_i-w\in\Vplus$, $i=1,2$.
Then $z= \max(u_1-w,u_2-w)$ is 
the $\Vplus$-least upper  bound of $u_1-w$ and $u_2-w$,
by Theorem~\ref{thm-char-lattice}, as $\Vplus$ is assumed to be a
lattice. 
Any upper bound in $V$ of $u_1-w$ and $u_2-w$ necessarily belongs to 
$\Vplus$, and thus $z$ is also 
least among all upper  bounds of $u_1-w$ and $u_2-w$ in $V$.
It follows that $w+z$ is a $V$-least upper  bound of $u_1$ and $u_2$.
As $w+z \in V_0$, it is 
least among all upper bounds in $V_0$ as well.
Applying this to $-u_1$ and $-u_2$ shows that $u_1$ and $u_2$ also have
a $V_0$-greatest lower bound, and thus $V_0$ is a lattice.

Next, assume that $V_0$ is a lattice and let $u_1,u_2 \in \Vplus$.
Then $u_1$ and $u_2$ have a $V_0$-least upper bound $w$ and 
a $V_0$-greatest lower bound $z$.
Any upper bound of $u_1$ (in $V$)
must belong to $\Vplus$ and thus $w$
is a $\Vplus$-least upper bound of $u_1$ and $u_2$.
Moreover, $0$ is a common lower bound of $u_1$ and $u_2$,
so $z \ge 0$, as it is 
greatest, and thus $z$ is a $\Vplus$-greatest lower bound.
Hence  $\Vplus$ is a lattice.

Finally, let $u_1, u_2 \in \Vplus$ and assume that $\Vplus$ (or equivalently
$V_0$) is a lattice.
By the above, the $\Vplus$-least upper (greatest lower) bound 
of $u_1$ and $u_2$ is 
$V_0$-least 
(greatest) as well.
It follows that it is also $V$-least (greatest), since every
upper (lower) bound $v\in V$ of $u_1$ and $u_2$ satisfies
$v\ge u_1\ge0$ ($v\le u_1\in\Vplus$) and thus necessarily belongs to $V_0$.
\end{proof}

\begin{proof}[Proof of Theorem~\ref{thm-V-lattice}]
First, assume that $V$ is a lattice.
Then $u \in V$ and $0$ have a $V$-least upper bound $w$ which must belong to $\Vplus$,
so $u \in V_0$, i.e.\ $V_0=V$. 
Moreover, $V_0$ is thus a lattice by assumption, and hence, 
by Proposition~\ref{prop-V0-lattice}, $\Vplus$ is also a lattice.

Conversely, if $\Vplus$ is a lattice and $V=V_0$, then $V=V_0$ is also a lattice
by Proposition~\ref{prop-V0-lattice} again.
\end{proof}

\section{Examples of gradient spaces. I. Sobolev spaces}
\label{sect-ex-1}

\noindent In this and the next section we shall present a collection of examples
that serve as a justification for the abstract gradient spaces we have
introduced. As we will see, several known cases of generalized gradients
are included and, in some cases, our framework leads to extensions of
previously known results, as well as a few new results. 
We also show that higher-order operators, such as the Laplacian,
can be considered as ``gradients'' in our  framework. 
Moreover, we
demonstrate explicitly that noncommutative algebras can be included,
by studying finite-dimensional matrix algebras as well as 
infinite-dimensional operator algebras. 

\subsection{Sobolev spaces on unweighted 
\texorpdfstring{$\R^n$}{Rn}}
\label{sect-unweighted-Rn}

Let $\MM(\R^n)$ and $\MM(\R^n,\R^n)$ be the sets of a.e.-equivalence classes
(with respect to the Lebesgue measure) of measurable functions from
$\R^n$ to $\R$ and from $\R^n$ to $\R^n$, respectively. Let $1<p<\infty$.
We introduce a relation $R$ so that
\begin{equation}   \label{eq-grad-spc-Rn}
    \U=(L^p(\R^n),\MM(\R^n), L^p(\R^n,\R^n),\MM(\R^n,\R^n),R)
\end{equation}
is a gradient space, where $L^p(\R^n,\R^n)$ is the set of 
vector-valued $L^p$ functions from $\R^n$ to $\R^n$.
Here we let
\[
    (u,\grad u) \in R
\]
if $\grad u$ is the distributional gradient of $u$, defined by
\begin{equation}  \label{eq-def-distrib-grad}
\int_{\R^n} u(x) \grad\phi(x) \,dx = -\int_{\R^n} \grad u(x) \phi(x) \,dx
\quad \text{for all }\phi\in C_0^\infty(\R^n),
\end{equation}
where  $C_0^\infty(\R^n)$ is the space of all infinitely
differentiable functions with compact support in $\R^n$.
Then $\Sob(\U)$ becomes the usual Sobolev space $W^{1,p}(\R^n)$
and since $L^p(\R^n)$ (for $1<p<\infty$) is uniformly convex, properties
\ref{gs:V.reflexive} and \ref{gs:W.reflexive.convex} hold. 
It follows immediately 
from~\eqref{eq-def-distrib-grad} that $\nabla (-u)=-\nabla u$,
and thus \ref{gs:minus.gradient} holds.
It is also straightforward that~\eqref{eq-def-distrib-grad} is preserved
under taking limits as $u_j\to u$ in $L^p(\R^n)$ and $\grad u_j\to v$
in $L^p(\R^n,\R^n)$, i.e.\ $(u,v)\in R$. 
Hence, \ref{gs:upper.gradient.limit} is fulfilled and we conclude
that $\U$ is a gradient space.  

A (linear) partial order on $\MM(\R^n)$ is introduced by writing $f\geq 0$ if
the set $\{x\in\R^n: f(x)<0\}$ has Lebesgue measure zero.
One also easily verifies that the requirements in 
Definitions~\ref{def:preordered.upper.gradient.space}
and~\ref{def:ordered-grad-space} are satisfied.
We may thus conclude that~\eqref{eq-grad-spc-Rn}
is an ordered gradient space with respect to the
standard (almost everywhere) ordering of functions.

It is well known that if $\Om\subset\R^n$ is a bounded open set then
\[
\|u\|_{L^p(\Om)} \le C_\Om \|\grad u\|_{L^p(\Om)}
\quad \text{for all } u\in C^\infty_0(\Om),
\]
see e.g.\ Theorem~2.4.1 in Ziemer~\cite{Ziemer}.
Thus, the closure $W^{1,p}_0(\Om)$ of $C^\infty_0(\Om)$ in $W^{1,p}(\R^n)$
is a Poincar\'e set.
By Theorem~2.5.1 in~\cite{Ziemer}, it is also a regular Rellich--Kondrachov cone.

With respect to the above gradient relation, the corresponding
Dirichlet problem becomes the \p-Laplace equation
\[
\Delta_p u := \dvg(|\grad u|^{p-2} \grad u) =0
\]
 with the boundary data
$u=f\in W^{1,p}(\R^n)$ on $\bdy\Om$, while the Rayleigh quotient gives rise
to the eigenvalue problem $\Delta_p u - \lambda u^{p-1}=0$
for $0\le u\in W^{1,p}_0(\Om)$.

Another possibility is to let
\[
    \Ut=(L^p(\R^n),\MM(\R^n), L^p(\R^n),\MM(\R^n),\widetilde{R})
\]
with $(u,g) \in \widetilde{R}$ if $g\ge|\grad u|$
a.e.\ (note that we need an inequality for property \ref{g:sum} to hold).
This leads to the
same Sobolev space $\Sob(\Ut)=W^{1,p}(\R^n)$.  It is easily seen that
all our axioms are satisfied also in this case and that we obtain an
ordered gradient space (with almost everywhere ordering).  This
definition of gradient relation is less orthodox, but is more in line
with the metric space definitions given in
Sections~\ref{sec:pweak.upper.gradient}--\ref{sec-grad-from-PI} below.
Let us however first consider weighted $\R^n$.

\subsection{Weighted \texorpdfstring{$\R^n$}{Rn}}
\label{sect-weighted-Rn}

Let $1<p<\infty$ and let 
$d\mu=w(x)\,dx$ be a \p-admissible weight, 
see Heinonen--Kilpel\"ainen--Martio~\cite{HeKiMa}.
Here we let 
\[
    \U=(L^p(\R^n,\mu),\MM(\R^n), L^p(\R^n,\R^n,\mu),\MM(\R^n,\R^n),R)
\]
with $(u,v) \in R$ if $v$ 
is the gradient of $u$ defined in
Section~1.9 in \cite{HeKiMa}, i.e.\ there is a sequence 
$\phi_i\in C^\infty(\R^n)$ such that both $\phi_i\to u$ in $L^p(\R^n,\mu)$
and $\grad\phi_i\to v$ in $L^p(\R^n,\R^n,\mu)$.
(The gradient depends on $p$ and $\mu$, but for 
locally Lipschitz functions it
coincides with the usual distributional gradient.)
We thus obtain $\Sob(\U)=H^{1,p}(\R^n;\mu)$, in the notation of \cite{HeKiMa}.
If $\mu$ is the Lebesgue measure, then $H^{1,p}(\R^n;\mu)$ is the unweighted
Sobolev space $W^{1,p}(\R^n)$ from Section~\ref{sect-unweighted-Rn}.
A natural partial order on $\Sob(\U)$ is as before given by $\mu$-a.e.\
pointwise inequality of functions.

As in the unweighted case, for a bounded open set $\Om\subset\R^n$, 
the closure $H^{1,p}_0(\Om;\mu)$ 
of $C^\infty_0(\Om)$ in $H^{1,p}(\R^n;\mu)$ will be a Poincar\'e set and
a regular Rellich--Kondrachov cone.
The corresponding Dirichlet problem will be
\[
\min \int_\Om |\grad u|^p w(x)\,dx
\]
among all $u\in H^{1,p}(\R^n;\mu)$ with $u-f\in H^{1,p}_0(\Om;\mu)$.
Here, $\grad u$ stands for the gradient defined above. 
One easily verifies that the minimization problem is equivalent to
the weighted \p-Laplace equation
$\dvg(w(x)|\grad u|^{p-2}\grad u)=0$, see Chapter~5 in~\cite{HeKiMa},
and that the Rayleigh quotient leads to the weighted eigenvalue problem
\(
\dvg(w(x)|\grad u|^{p-2}\grad u)-\lambda u^{p-1}w(x)=0.
\)

As in the unweighted situation, we can alternatively let
\begin{equation} \label{eq-grad-struct-alt}
    \Ut=(L^p(\R^n,\mu),\MM(\R^n), L^p(\R^n,\mu),\MM(\R^n),\widetilde{R})
\end{equation}
with $(u,g) \in \widetilde{R}$ if $g\ge|\grad u|$ a.e.
Then $\Sob(\Ut)=H^{1,p}(\R^n;\mu)$ as before,
and the above minimization problems  will be the same
as in the vector-valued case.
One can also consider similar gradient spaces on open subsets of
$\R^n$.  In all cases it is easily seen that all our axioms are
satisfied and that we obtain ordered gradient spaces.

Another nonstandard choice of a gradient relation in unweighted $\R^n$
is
\[
\widehat{R} = \{(u,g) : u\in W^{1,p}(\R^n) \text{ and } 
    g \ge M|\grad u| \text{ a.e.}\},
\]
where  $M|\grad u|$ is the Hardy--Littlewood maximal function 
of $\grad u$.
This choice is related to the Haj\l asz gradient in 
Section~\ref{sec:Hajlasz.gradient} below and leads to yet another Dirichlet
problem.
Similarly, on weighted $\R^n$, the weighted maximal function
\[
M_{\mu}v(x) := \sup_{B\ni x} \frac{1}{\mu(B)} \int_B |v|\, d\mu
\] 
provides a gradient relation.

For other possible choices of gradient relations
leading to different minimization
problems and partial differential equations see 
Section~\ref{sect-gen-var-probs} below.

\subsection{Newtonian Sobolev spaces on metric spaces}
\label{sec:pweak.upper.gradient}

In order to see how to fit these spaces 
into our theory of gradient spaces we first 
need to give a (very brief) introduction to the theory
of Newtonian spaces.

Let $1<p<\infty$ and let $X=(X,d,\mu)$ be a 
metric space equipped with a metric $d$ and a Borel
regular measure $\mu$, which is positive and finite on all balls.
A measurable function $g:X\to[0,\infty]$ is a \emph{\p-weak upper gradient} 
of an everywhere defined  function
$u:X\to[-\infty,\infty]$ if for 
\p-almost all 
nonconstant rectifiable curves $\ga$ in $X$, 
\begin{equation}   \label{def-upper-grad}
|u(x)-u(y)| \le \int_\ga g\,ds,
\end{equation}
where $x$ and $y$ are the endpoints of $\ga$, and
the integration over $\ga$ is with respect to the arc length $ds$.
Here \p-almost all means 
that there exists $\rho\in L^p(X)$ such that $\int_\ga\rho\,ds=\infty$ 
for all curves $\ga$ failing~\eqref{def-upper-grad},
see e.g.\ Chapter~1 in Bj\"orn--Bj\"orn~\cite{BBbook}
for this and the other basic Newtonian theory needed here.

Moreover, a Borel function $g:X\to[0,\infty]$ is an \emph{upper gradient}
of $u$ if \eqref{def-upper-grad} holds for \emph{all} nonconstant rectifiable curves.
However, if we had based our theory upon upper gradients, then
\ref{gs:upper.gradient.limit} would have failed, 
and therefore the \p-weak upper gradients will be the ones that
are of interest to us.

Upper gradients were introduced by Heinonen and Koskela~\cite{HeKo98},
while Koskela and MacManus~\cite{KoMc} introduced the \p-weak upper gradients
(because of the problem mentioned above with upper gradients and 
\ref{gs:upper.gradient.limit}).
Shanmugalingam~\cite{Sh-rev} defined
the Newtonian spaces based on these notions (they can equivalently be 
defined using 
either upper gradients or \p-weak upper gradients)
and showed that they are always Banach spaces.

Note that $u$ above is required 
to be everywhere defined for the definition of (\p-weak) upper gradients
to make sense. 
Indeed, it is easily verified that 
if $\ut=u$ $\mu$-a.e.\ and $g$ is a \p-weak upper gradient of $u$, then 
$g$ need not be a \p-weak upper gradient of $\ut$.
(In $\R^2$, let e.g.\ $u\equiv0$ and $\ut=\chi_{\R \times \{0\}}$,
i.e.\ the characteristic function of $\R \times \{0\}$.)
However, as we want $V$ and $W$ to be normed (not just seminormed) spaces, 
the easiest approach is to consider $\mu$-a.e.\ equivalence classes.

We therefore let $\MM(X)$ be the set of 
$\mu$-a.e.\ equivalence classes of measurable functions from
$X$ to $[-\infty,\infty]$, and set
\[
       \U=(L^p(X),\MM(X),L^p(X),\MM(X),R),
\]
where $([u],[g]) \in R$ if there are everywhere defined representatives
$\ut \in [u]$ and $\gt \in [g]$ such that $\gt$ is a \p-weak upper gradient
of $\ut$.
Note, however, that not all representatives of $[u]$ have \p-weak upper
gradients in $L^p(X)$.

The space  $\Sob(\U)$ becomes the Newtonian Sobolev space. 
Strictly speaking $\Sob(\U)=\hNp(X)/{\sim}$ in the notation of \cite{BBbook}
(where $\sim$ is the $\mu$-a.e.-equivalence relation),
while the Newtonian space $\Np(X)$ considered therein consists of all 
pointwise defined 
$u \in L^p(X)$ which have \p-weak upper gradients in $L^p(X)$.
If $X=\R^n$ (with Euclidean distance) and 
$d\mu=w\,dx$ is a \p-admissible weight, 
then $\Sob(\U)$ 
coincides
with $H^{1,p}(\R^n;\mu)$ from Section~\ref{sect-weighted-Rn},
see Appendix~A.2 in \cite{BBbook}.
If $\mu$ is the Lebesgue measure, then $\Sob(\U)=W^{1,p}(\R^n)$,
see Shanmugalingam~\cite{Sh-rev}.

That axioms \ref{g:sum}, \ref{g:mult.scalar} and
\ref{gs:V.reflexive}--\ref{gs:minus.gradient} hold follows from
results in Chapter~1 of \cite{BBbook}, while
\ref{gs:upper.gradient.limit} follows from Proposition~2.3 in
\cite{BBbook}.  It is also clear (due to the monotonicity of
integration) that $\U$ is an ordered gradient space and 
a lattice (with the $\mu$-a.e.\ ordering).  
Moreover, the maximum $\max(u,v)$ is then the
$\mu$-a.e.\ pointwise maximum.

The minimal gradient that we obtained in
Theorem~\ref{thm:DpU.minimal.gradient} is in this case usually
called the minimal \p-weak upper gradient and it is not only
norm-minimal but also pointwise minimal $\mu$-a.e.\ and local in the
sense that it is zero $\mu$-a.e.\ on every level set $\{x\in X:
u(x)=c\}$, see Chapter~2 in \cite{BBbook}.

Standard assumptions in the theory of Newtonian spaces on metric
spaces are that the underlying measure $\mu$ is \emph{doubling} and
supports a \emph{\p-Poincar\'e inequality}, i.e.\ there exist
constants $C>0$ and $\lambda \ge 1$ such that for all balls $B=B(x,r)$
we have $\mu(2B)\le C\mu(B)$, and for all integrable functions $u$ on
$X$ and all (\p-weak) upper gradients $g$ of $u$,
\begin{equation}   \label{eq-def-p-PI}
        \vint_{B} |u-u_B| \, d\mu
        \le C r \biggl( \vint_{\lambda B} g^{p} \,d\mu \biggr)^{1/p},
\end{equation}
where $ f_B :=\vint_B f \,d\mu := \int_B f\, d\mu/\mu(B)$
and $\lambda B= B(x,\lambda r)$.
Here, for simplicity, we also assume that $X$ is unbounded.
These assumptions imply 
that the Friedrichs' inequality
\begin{equation} \label{eq-Friedrichs}
        \int_{E} |u|^p \,d\mu \le C_E \int_{E} g_u^{p} \,d\mu
\end{equation}
holds for all $u$ in 
\begin{equation}   \label{eq-Nhat-wrong}
   \widehat{\K}_0(E):=\{[u] \in \Sob(\U): u=0 \text{ on } X \setm E\},
\end{equation}
where $E$ is an arbitrary bounded measurable subset of $X$.
Thus, the space $\widehat{\K}_0(E)$ 
is a (closed convex) Poincar\'e set for any bounded measurable set
$E\subset X$,
and the theory from Sections~\ref{sect-grad-spcs}--\ref{sec-obstacle-problem}
can be applied.

Usually one considers the slightly smaller zero-Sobolev space
\begin{equation} \label{eq-hNp0}
   \K_0(E)
   :=\{[u]: u \in \Np(X) \text{ and }
   u=0 \text{ on } X \setm E\} \subset \widehat{\K}_0(E).
\end{equation}
If $G$ is open, and $X$ is weighted $\R^n$ (as in Section~\ref{sect-weighted-Rn}),
then $\K_0(G)=H^{1,p}_0(G;\mu)$,
while $\widehat{\K}_0(G)$ can be larger;
e.g.\ if $G=B \setm K$, where $B$ is a 
ball and $K \subset B$ is a compact set
with zero measure and positive capacity, then 
\[
   \widehat{\K}_0(G) = H^{1,p}_0(B;\mu)
   \varsupsetneq H^{1,p}_0(G;\mu) = \K_0(G).
\]
For nonopen $E$ the space $\K_0(E)$ is essentially the space
$\Np_0(E)$, cf.\ the study of Dirichlet and obstacle problems on
nonopen sets in Bj\"orn--Bj\"orn~\cite{BBnonopen}.  The Dirichlet
problem in this setting was first considered by
Shanmugalingam~\cite{Sh-harm}, and extensively studied since then.
See Bj\"orn--Bj\"orn~\cite{BBbook} for further references.

Since the minimal \p-weak upper gradient is not a linear operation
(as $g_{-u}=g_u$), Proposition~\ref{prop:u.gu.linear.unique.solution}
cannot be applied in this case.
Nevertheless, uniqueness of solutions of the Dirichlet problem is 
proved by different methods under the
assumption of a \p-Poincar\'e inequality, see 
Cheeger~\cite[Theorem~7.14]{Cheeg}
or Bj\"orn--Bj\"orn~\cite[Theorem~7.2]{BBnonopen}.

It is also known that the \p-Poincar\'e inequality implies the Sobolev
embedding $\widehat{\K}_0(E) 
\to L^q(E)$ for some $q>p$, and that the embedding $\widehat{\K}_0(E)
\to L^p(E)$ is compact, see
Bj\"orn--Bj\"orn~\cite[Theorem~5.51]{BBbook} and Haj\l
asz--Koskela~\cite[Theorems~5.1 and~8.1]{HaKo2000}.  Thus, $\K_0(E)$ and
$\widehat{\K}_0(E)$ are regular Rellich--Kondrachov cones, which makes
it possible to also solve the ``eigenvalue problems''
\[
\min_{u\in\K_0(E)} \frac{\|g_u\|_{L^p(X)}}{\|u\|_{L^p(X)}}
\quad \text{and} \quad
\min_{u\in\widehat{\K}_0(E)} \frac{\|g_u\|_{L^p(X)}}{\|u\|_{L^p(X)}},
\]
see e.g.\ Latvala--Marola--Pere~\cite{LMP}.

\subsection{Newtonian spaces based on Banach function lattices}

In the above theory of Newtonian spaces one can replace the $L^p$ space
by another space. 

A vector space $Y \subset \MM(X)$ of ($\mu$-a.e.\ equivalence classes of) 
measurable functions
on a  metric space $(X,d,\mu)$ 
(as in Section~\ref{sec:pweak.upper.gradient})
is a \emph{Banach function lattice} if
the following axioms hold:
\begin{enumerate} 
\item[(P0)]
$\normY{\cdot}$ determines $Y$, i.e.\ $Y=\{u \in \MM(X) : \|u\|_Y<\infty\}$;
\item[(P1)]
$\normY{\cdot}$ is a norm;
\item[(P2)]
the lattice property holds, i.e.\
if $|u| \le |v|$ $\mu$-a.e.\, then $\|u\|_Y \le \|v\|_Y$;
\item[(RF)]
the Riesz--Fischer property holds, i.e.\
if $u_n \ge 0$ $\mu$-a.e.\ for all $n$, then 
\[ 
      \biggl\|\sum_{n=1}^\infty u_n \biggr\|_Y \le \sum_{n=1}^\infty \|u_n\|_Y.
\]
\end{enumerate} 
Note that a Banach function space is a more restrictive concept requiring
some further axioms, see the discussion
in Mal\'y~\cite{lmaly1}.

Let us additionally assume that $Y$ is reflexive, and that the
norm is strictly convex.  We then introduce the gradient space
\[
    \U=(Y,\MM(X),Y,\MM(X),R),
\]
where $([u],[g]) \in R$ if there are representatives
$\ut \in [u]$ and $\gt \in [g]$ such that $\gt$ is a $Y$-weak upper gradient
of $\ut$, i.e.\ that \eqref{def-upper-grad} holds for $Y$-almost all curves $\ga$.
For the exact definition of $Y$-weak upper gradients see \cite{lmaly1}, 
where the basic theory is developed. 
That axioms 
\ref{g:sum}, \ref{g:mult.scalar} and
\ref{gs:V.reflexive}--\ref{gs:minus.gradient}
hold also follows from results in \cite{lmaly1},
while \ref{gs:upper.gradient.limit} follows from Proposition~5.6
in Mal\'y~\cite{lmaly2}.
Moreover, the gradient space is ordered (with the $\mu$-a.e.\ ordering).

The minimal gradient that we obtained
 in Theorem~\ref{thm:DpU.minimal.gradient}
is in this case usually called the minimal $Y$-weak upper gradient.
Also here it is pointwise minimal $\mu$-a.e.\ and local,
see \cite{lmaly2}.
Note however that in \cite{lmaly2} and \cite{lmaly1}, $Y$ is not assumed
to be reflexive and the norm $\normY{\cdot}$ need not be strictly convex.

One can define $\K_0(E)$ and $\widehat{\K}_0(E)$
similarly as in \eqref{eq-Nhat-wrong} and \eqref{eq-hNp0}.
They are closed, by e.g.\ Corollary~7.2 in \cite{lmaly1}, and
clearly convex.
Whenever they are Poincar\'e sets
 our results in Sections~\ref{subsect-DP} and~\ref{sec-obstacle-problem}
show the existence of solutions to the Dirichlet and obstacle problems.
(We are not aware of any earlier such results in the Newtonian theory beyond $L^p$ 
spaces.)

\subsection{The Haj{\l}asz gradient}\label{sec:Hajlasz.gradient}

The theory of Haj\l asz gradients, as introduced by Haj\l asz in \cite{Haj-PA},
is similar to that of (\p-weak) upper gradients
in that it also applies to arbitrary metric spaces and that the gradient is
scalar (rather than vector-valued) and the operation of minimal Haj\l asz gradient
is not linear.
There are, however, several substantial differences,
such as lack of locality, see below.

Let  $(X,d,\mu)$ be a metric space as in
Section~\ref{sec:pweak.upper.gradient}. A Borel function
$h:X\to[0,\infty]$ is a \emph{Haj\l asz gradient} of a function
$u:X\to[-\infty,\infty]$ if for all $x,y\in X\setm Z$, where
$\mu(Z)=0$, we have that
\begin{equation}   \label{def-Hajlasz-grad}
|u(x)-u(y)| \le d(x,y) (h(x)+h(y)).
\end{equation}
As before, to obtain an ordered gradient space we choose $1<p<\infty$
and let
\[
   \U=(L^p(X),\MM(X),L^p(X),\MM(X),R).
\]
This time  $(u,h) \in R$ if $h$ is a Haj\l asz gradient of $u$
(which is easily seen to be well-defined on $\mu$-a.e.\ equivalence classes
of functions).

The Sobolev space $\SobU$ then coincides with the Haj\l asz space $M^{1,p}(X)$,
consisting of all $u\in L^p(X)$ with a Haj\l asz gradient in $L^p(X)$.
This almost always (except for some pathological situations) makes $\SobU$ to
be strictly smaller that $L^p(X)$.
In contrast, if $X$ contains no nonconstant rectifiable curves (e.g.\  
$X=\R\setm\Q$ or $X=\R$ with the snowflaked metric $d(x,y)=|x-y|^\al$, $0<\al<1$),
then $\hNp(X)/{\sim}$ equals $L^p(X)$.
Lemma~4.7 and Theorem~4.8 in Shanmugalingam~\cite{Sh-rev} 
show that one always has 
$M^{1,p}(X)\subset\hNp(X)/{\sim}$ and that 
$4$ times a Haj\l asz gradient is an upper gradient.
Moreover, if $\mu$ is doubling and $X$ satisfies a 
$q$-Poincar\'e inequality for upper gradients for some $q<p$, then
$M^{1,p}(X)=\hNp(X)/{\sim}$.
(Note that if $X$ is complete, $\mu$ is doubling and $X$ satisfies
an (upper gradient) \p-Poincar\'e inequality, then by Keith--Zhong~\cite{KZ}
it also satisfies a $q$-Poincar\'e inequality for some $q<p$.)

It is easily verified that axioms \ref{g:sum}, \ref{g:mult.scalar} and
\ref{gs:V.reflexive}--\ref{gs:minus.gradient} hold for  Haj\l asz gradients.
That \ref{gs:upper.gradient.limit} is satisfied follows from the fact that
every $L^p$-convergent sequence has a $\mu$-a.e.\ converging subsequence.
Moreover, the gradient space is ordered (with the $\mu$-a.e.\ ordering).
Thus, the results in Section~\ref{sect-grad-spcs} apply to Haj\l asz spaces
and show in particular that they are Banach spaces and that there is 
a unique minimal Haj\l asz gradient.
These properties were originally proved in Theorems~2 and~3 in~\cite{Haj-PA}.

Next, Haj\l asz functions \emph{always} satisfy the 1-Poincar\'e inequality, as
is easily seen by twice integrating~\eqref{def-Hajlasz-grad} over a ball $B$
with respect to $x$ and $y$.
As for the Friedrichs' inequality~\eqref{eq-Friedrichs}, a similar repeated 
integration over $B$ and $2B\setm B$ gives for all $u\in M^{1,p}(X)$
vanishing outside a ball $B=B(x,r)$,
\begin{align*}
\vint_B |u|^p\,d\mu &= \vint_{2B\setm B} \vint_B |u(x)-u(y)|^p\, d\mu(x) \,d\mu(y) \\
&\le  \vint_{2B\setm B} \vint_B d(x,y)^p (h(x)+h(y))^p\, d\mu(x) \,d\mu(y) \\
&\le C' 
   r^p \biggl( \vint_B h(x)^p\,d\mu(x) + \vint_{2B\setm B} h(y)^p\,d\mu(y) \biggr)\\
&\le C'' r^p \vint_{2B} h^p\,d\mu,
\end{align*}
provided that $\mu$ is doubling and reverse doubling for $B$, i.e.\ that 
\[
(1+\eps)\mu(B) \le \mu(2B)\le C\mu(B)
\]
for some $\eps,C>0$. 
Thus,
\[
   \|u\|_{L^p(X)}^p  = \int_B |u|^p\,d\mu 
   \le    C_B \int_{2B} h^p\,d\mu
   \le   C_B \|h\|_{L^p(X)}^p .
\]
Note here that the last integral cannot be taken only over $B$, since the Haj\l asz
gradient lacks locality, i.e.\ $u=0$ in some set does not imply that the minimal
Haj\l asz gradient $h_u=0$ in that set.
On the other hand, locality
holds for minimal weak 
upper gradients based on curves.
Also contrary to the upper gradient case, the minimal Haj\l asz gradients
are only norm-minimal, and not $\mu$-a.e pointwise minimal, see Example~B.1
in \cite{BBbook}.

Nevertheless,  our results in Sections~\ref{subsect-DP}
and~\ref{sec-obstacle-problem}
show the existence of solutions to the Dirichlet and obstacle problems.
As far as we know, these problems have not been considered for the Haj\l asz
spaces before.

\subsection{Gradients from Poincar\'e inequalities}
\label{sec-grad-from-PI}

Another possibility to define a gradient relation is through Poincar\'e inequalities;
namely, we say that $(u,k)\in R$ if for all balls $B=B(x,r)\subset X$,
\begin{equation}     \label{eq-def-grad-PI}
    \vint_{B} |u-u_B| \, d\mu \le r \vint_{\lambda B} k \,d\mu,
\end{equation}
where $u_B$ is the integral average as before,
and $\lambda \ge 1$ is some fixed constant. 
We also let 
\[
   \U=(L^p(X),\MM(X),L^p(X),\MM(X),R),
\]
$1<p<\infty$.

It is clear that \ref{g:sum} and \ref{g:mult.scalar} hold.
(Note, however, that for \ref{g:mult.scalar} it is important that we use the 
1-Poincar\'e inequality in~\eqref{eq-def-grad-PI},
since with a \p-Poincar\'e inequality in~\eqref{eq-def-grad-PI}
the subadditivity is not at all clear.) 
Also, \ref{gs:V.reflexive}--\ref{gs:upper.gradient.limit} are clearly 
satisfied, since~\eqref{eq-def-grad-PI} is preserved under taking $L^p$-limits.

The space $\SobU$ consists of all $u\in L^p(X)$ such that \eqref{eq-def-grad-PI}
holds for some \emph{Poincar\'e gradient} $k\in L^p(X)$.
Theorem~\ref{thm:DpU.minimal.gradient} then implies that there exists a
$\mu$-a.e.\ 
unique minimal Poincar\'e gradient $k_u$ satisfying~\eqref{eq-def-grad-PI}.

To solve the Dirichlet and obstacle problems we need Poincar\'e sets.
An application of H\"older's inequality to~\eqref{eq-def-grad-PI}
implies the \p-Poincar\'e inequality~\eqref{eq-def-p-PI} (with $g$
replaced by $k$). 
Assume that $\mu$ is doubling.
We can now use Theorem~5.1 from Haj\l asz--Koskela~\cite{HaKo2000} which implies
the $(q,p)$-Poincar\'e inequality
\begin{equation}     \label{eq-(p,p)-PI}
       \biggl(  \vint_{B} |u-u_B|^q \, d\mu \biggr)^{1/q}
        \le C r \biggl( \vint_{5\lambda B} k_u^{p} \,d\mu \biggr)^{1/p}
\end{equation}
for some $q>p$, and in particular the $(p,p)$-Poincar\'e inequality.
Now let $E\subset X$ be a bounded and measurable set, and fix a ball $B$ such that
$E\subset B$ and $\mu(B\setm E)>0$.
Let $\K_0$ consist of all $u\in\SobU$ such that $u=0$ $\mu$-a.e.\
in $X\setm E$. 
Then \eqref{eq-(p,p)-PI} with $q=p$ implies that for every $u\in\K_0$, 
\begin{align}  \label{eq-to-show-Friedrichs}
\biggl(  \vint_{B} |u|^p \, d\mu \biggr)^{1/p} 
    &\le \biggl(  \vint_{B} |u-u_B|^p \, d\mu \biggr)^{1/p} + |u_B| \nonumber\\
    &\le C r \biggl( \vint_{5\lambda B} k_u^{p} \,d\mu \biggr)^{1/p} + |u_B|.
\end{align}
Moreover, H\"older's inequality and the fact that $u=0$ $\mu$-a.e.\ 
outside $E$ yield
\[
|u_B| \le \vint_B |u|\chi_E \,d\mu 
     \le \biggl( \vint_{B} |u|^p \, d\mu \biggr)^{1/p} 
                \biggl( \frac{\mu(E)}{\mu(B)} \biggr)^{1-1/p} 
     \le \theta \biggl(  \vint_{B} |u|^p \, d\mu \biggr)^{1/p}, 
\]
where $\theta\in(0,1)$.
Inserting this into~\eqref{eq-to-show-Friedrichs} and subtracting from both sides
yields
\[
(1-\theta) \biggl(  \vint_{B} |u|^p \, d\mu \biggr)^{1/p} 
   \le C r \biggl( \vint_{5\lambda B} k_u^{p} \,d\mu \biggr)^{1/p}.
\]
{}From this and the doubling property of $\mu$ we conclude that
\[
\|u\|^p_V = \int_B |u|^p\,d\mu
   \le C_E \int_{5\lambda B} k_u^{p} \,d\mu \le C_E \|k_u\|^p_W,
\]
i.e.\ that $\K_0$ is a Poincar\'e set.
Proposition~\ref{prop:u.gu.linear.unique.solution} now makes it possible to solve
the Dirichlet problem $\min\|k_u\|_W$ among all $u\in\SobU$ such that $u=f$
in $X\setm E$ and $f\in\SobU$ is fixed. 
That $\K_0$ is a Rellich--Kondrachov cone is guaranteed by Theorem~8.1 
in~\cite{HaKo2000}.
We remark that neither the Dirichlet nor the obstacle problem have been
studied in this setting before.

\section{Examples of gradient spaces. II. Other examples}
\label{sect-ex-2}

\subsection{Gradient spaces from continuous linear maps}
\label{sec:ugs.from.continuous}

\noindent Let $\Vt$ and $\Wt$ be vector spaces, let $ V\subseteq \Vt$
be a reflexive Banach space and let $W\subseteq \Wt$ be a strictly
convex Banach space. Let $D\subseteq V$ be a (norm)-closed linear
subspace and let $F:D\to \Wt$ be a continuous linear map.  One defines
the graph of $F$ as the relation $R_F\subseteq \Vt\times \Wt$ with
\begin{align*}
  R_F = \{(u,F(u)): u\in D\},
\end{align*}
and one may readily check that $R_F$ is a gradient relation. To
show that $\U=(V,\Vt,W,\Wt,R_F)$ is a gradient space, one
needs to check property
\ref{gs:upper.gradient.limit} (note that
\ref{gs:minus.gradient} holds since $F$ is linear). 
Thus, assume that $u,u_i\in V$ and $g,g_i\in W$ are such that
$(u_i,g_i)\in R_F$ (i.e.\ $g_i=F(u_i)$) for $i=1,2,\ldots$, and that
\begin{align*}
  \normV{u-u_i}\to 0\textrm{ and }\normW{g-F(u_i)}\to 0.
\end{align*}
Since the domain of $F$ is assumed to be closed, it follows that $u\in
D$, which implies that
\begin{align*}
  \norm{F(u)-F(u_i)}_W=\norm{F(u-u_i)}_W\to 0,
\end{align*}
as $F$ is continuous. Consequently,  $g=F(u)$ which,
by definition, implies that $(u,g)\in R_F$. 

We shall now provide two more concrete examples of this approach, together
with some applications to partial differential equations.

\subsection{More general variational problems and PDEs.}
\label{sect-gen-var-probs}

Let $\Om\subset\R^n$ be a bounded open set and consider $D=V=W^{1,p}(\Om)$
and $W=L^p(\Om;\R^n)$, $1<p<\infty$, e.g.\ with the norm 
$\|v\|^p_W = \int_\Om |v|^p\,dx$.
Let $A(x)$ be an $(n\times n)$-matrix with bounded real measurable entries,
which is uniformly elliptic in the sense that $|A(x)\xi|\ge\al|\xi|$
for some $\al>0$, all $x\in\Om$ and all $\xi\in\R^n$.

Then $F:W^{1,p}(\Om)\to L^p(\Om;\R^n)$, given by $F(u)=A(x)\grad u$,
is a continuous linear mapping and defines a gradient relation $R_F$
as in Section~\ref{sec:ugs.from.continuous}.  The space $\SobU$
obtained in this way is the usual Sobolev space $W^{1,p}(\Om)$, and it
is naturally ordered by the a.e.\ 
ordering of functions.  Theorems~2.4.1 and~2.5.1 in
Ziemer~\cite{Ziemer}, together with the ellipticity of $A$, imply that
the subspace $W^{1,p}_0(\Om)$, which is the closure of
$C^\infty_0(\Om)$ in $W^{1,p}(\Om)$, is a regular Rellich--Kondrachov
cone (and thus a Poincar\'e set) also with respect to
the gradient relation $R_F$ (cf.\ Section~\ref{sect-unweighted-Rn}).
Hence, for every $f\in W^{1,p}(\Om)$,
Theorem~\ref{thm:dirichlet.problem} 
and Proposition~\ref{prop:u.gu.linear.unique.solution} provide us with a unique
solution of the Dirichlet problem
\[
\min \int_\Om |A(x)\grad u(x)|^p \,dx
\]
among all $u\in W^{1,p}(\Om)$ with $u-f\in W^{1,p}_0(\Om)$.
This minimizer is a weak solution of the elliptic equation
\[
\dvg( |A(x)\grad u|^{p-2} A(x)^T A(x)\grad u) =0.
\] 
In particular, if $A$ is the identity matrix, 
then this is the classical \p-Laplace equation
$\Delta_p u=0$.

Similarly, Theorem~\ref{thm-solve-rayleigh} makes it possible to
minimize the Rayleigh quotient 
\[
\min_{u\in W^{1,p}_0(\Om)} \frac{\int_\Om |A(x)\grad u|^p\,dx}{\|u\|^p_{L^p(\Om)}}.
\]

In this setting, there are other possible choices for $F$. 
For $V=D$ as above and $W=L^p(\Om)\times L^p(\Om;\R^n)$, equipped with 
the norm 
\[
\|(v,\xi)\|_W= \biggl( \int_\Om (|v|^p + |\xi|^p)\,dx \biggr)^{1/p},
\]
consider
\begin{equation}   \label{eq-F(u)-mix-u-gradu}
F(u)(x) = (\Lambda(u),\partial_1 u(x), \ldots,\partial_n u(x)) 
\in \R\times\R^n, 
\end{equation}
where $\Lambda:W^{1,p}(\Om) \to L^p(\Om)$ is any bounded linear operator.
The obtained space $\SobU$ is still $W^{1,p}(\Om)$ and its subspace
$W^{1,p}_0(\Om)$ is a regular Rellich--Kondrachov cone (and
thus a Poincar\'e set).
The Dirichlet problem then corresponds to the variational problem
\[
\min \int_\Om (|\Lambda(u)|^p + |\grad u|^p) \,dx,
\]
among all  $u\in W^{1,p}(\Om)$ with $u-f\in W^{1,p}_0(\Om)$.
Some natural choices are $\Lambda(u)=u$ and $\Lambda(u)=u\chi_E$, 
where $\chi_E$ is
the characteristic function of a measurable set $E\subset\Om$.
These choices lead to 
\[
\min \int_\Om (|u|^p + |\grad u|^p) \,dx
\quad \text{and} \quad
\min \biggl( \int_E |u|^p \,dx + \int_\Om |\grad u|^p \,dx \biggr).
\]
When $\Om=\R^n$, the obstacle problem, associated with the first functional and 
the obstacle $\psi=\chi_G$ for an open $G\subset\R^n$, is closely
related to the definition of the Sobolev capacity,  
see e.g.\ Section~2.35 in Heinonen--Kilpel\"ainen--Martio~\cite{HeKiMa}.

\subsection{Higher-order operators.}
\label{sec:higher.order.operators}

Let $V=D=W^{2,2}(\R^n)$ denote the subspace of $L^2(\R^n)$ 
  consisting of functions with first and
second order distributional derivatives in $L^2(\R^n)$.
Furthermore, let $W=L^2(\R^n)$ and $F(u)=\Delta u$, which clearly
is a continuous linear mapping from $V$ to $W$.

Theorem~4.4.1 in Ziemer~\cite{Ziemer} implies that for all 
$u\in W^{2,2}_0(\Om)$, which is the closure of $C^\infty_0(\Om)$ 
in $W^{2,2}(\R^n)$, the $W^{2,2}(\R^n)$-norm is 
equivalent to 
\[
\sum_{i,j=1}^n \int_\Om |\partial_{ij}u|^2\,dx.
\]
Since $\widehat{F(u)}=-|\xi|^2\uh$ and 
$\widehat{\partial_{ij}u}=-\xi_i \xi_j\uh$, where $\uh$ is the Fourier 
transform of $u$, Parseval's identity shows that
for all $u\in C^\infty_0(\Om)$, and hence for $u\in W^{2,2}_0(\Om)$,
\begin{align*}
\sum_{i,j=1}^n \int_\Om |\partial_{ij}u|^2\,dx 
     &= C \sum_{i,j=1}^n \int_{\R^n} |\uh|^2 |\xi_i \xi_j|^2\,d\xi \\
     &\le C' \int_{\R^n} |\uh|^2 |\xi|^4\,d\xi
     = C'' \|\Delta u\|^2_{L^2(\R^n)}.
\end{align*}
{}From this we conclude that $\|u\|_V\le C'''\|\Delta u\|_W$,
and it follows
that $W^{2,2}_0(\Om)$ is a Poincar\'e set with respect to the
``gradient'' relation 
\[
R=\{(u,\Delta u): u\in W^{2,2}(\R^n)\}.
\]
Theorem~\ref{thm:dirichlet.problem} 
and Proposition~\ref{prop:u.gu.linear.unique.solution}
now provide us, 
for every $f\in W^{2,2}(\R^n)$, with a unique solution of the Dirichlet problem
\[
\min \int_\Om |\Delta u|^2\,dx
\]
among all $u\in W^{2,2}(\R^n)$ with $u-f\in W^{2,2}_0(\Om)$.
In terms of partial differential equations, the above minimization problem 
is equivalent to the biharmonic equation $\Delta^2 u=0$ with the boundary 
data $u=f$ and $\grad u=\grad f$ on $\partial\Om$.

Of course, $F(u)=\Delta u$ is not the only choice. 
One may as well let $F(u)$ be the Hessian matrix of $u$.
Higher-order gradients and operators can also be considered.
It is also easy to mix derivatives of different orders in the same way 
as in~\eqref{eq-F(u)-mix-u-gradu}.
We will not dwell further upon these extensions here.

\subsection{Gradient spaces from lower semicontinuous 
sublinear  maps}
\label{sec:ugs.from.vector.valued}

\noindent Let $\Vt$ be a vector space and let $(\Wt,\leq)$ be a preordered
vector space. 
A map $F:D\to \Wt$ is  \emph{sublinear} if 
$D$ is a linear subspace and
\begin{enumerate}
\item $F(u+v)\leq F(u)+F(v)$,
\item $F(\alpha u)=\alpha F(u)$, 
\end{enumerate}
for all $u,v\in D$ and $\alpha>0$.

\begin{proposition}
  Let $F:D\to \Wt$ be a sublinear map. Then the set
  \begin{align*}
    R_F = \{(u,g) \in \Vt\times \Wt: F(u)\leq g\}
  \end{align*}
  is a gradient relation on $\Vt\times \Wt$.
\end{proposition}

\begin{proof}
  \ref{g:sum} Assume that $(u,g)\in R_F$ and $(u',g')\in
  R_F$, which implies that $F(u)\leq g$ and $F(u')\leq g'$. As $F$
  is sublinear, $F(u+u')\leq
  F(u)+F(u')\leq g+g'$, which implies that $(u+u',g+g')\in R_F$.

  \ref{g:mult.scalar} Assume that $(u,g)\in R_F$ and that
  $\alpha>0$. Since $f$ is a sublinear map it holds that
  $F(\alpha u)=\alpha F(u)\leq \alpha g$ which implies that $(\alpha
  u,\alpha g)\in R_F$.
\end{proof}

\noindent 
Let $ V\subseteq \Vt$ and $W\subseteq \Wt$ be reflexive Banach spaces
such that $W$ is strictly convex. Furthermore, we assume that $D$ is a
closed linear subspace of $V$, $F(D) \subset W$ and that $F:D \to W$ is
sublinear and lower semicontinuous.  This immediately implies that
properties \ref{gs:V.reflexive}--\ref{gs:minus.gradient} are
fulfilled.  Let us now show that property
  \ref{gs:upper.gradient.limit} holds. Assume that $D\ni
u_i\to u$ (in $V$) and $g_i\to g$ (in $W$) with $(u_i,g_i)\in R_F$ for
$i=1,2,\ldots$, which implies that $F(u_i)\leq g_i$ for
$i=1,2,\ldots$. Since $D$ is closed, it is clear that $u\in D$. Then,
as $F$ is lower semicontinuous, one obtains
\begin{align*}
  F(u)\leq\liminf_{i\to\infty} F(u_i) \leq 
  \liminf_{i\to\infty}g_i=\lim_{i\to\infty}g_i = g,
\end{align*}
which implies that $(u,g)\in R_F$. Hence, $(V,\Vt,W,\Wt,R_F)$ is a
gradient space. 

Note that the Newtonian spaces and their \p-weak upper gradients (as
well as the Haj\l asz gradients and the gradients given by Poincar\'e
type inequalities) can be seen as a special case of the above
construction.  We have also seen that in those cases there are plenty
of Poincar\'e sets and regular Rellich--Kondrachov cones.
Moreover, in the same spirit as in Section~\ref{sect-gen-var-probs},
the \p-weak upper gradients from
Section~\ref{sec:pweak.upper.gradient} can be combined with e.g.\ $u$
into new sublinear maps, such as the vector-valued map $u\mapsto
(u,g_u)$.

\subsection{Matrix algebras}
\label{sect-matrix-alg} 

So far, we have considered examples of gradient spaces based on
commutative algebras (of functions), as well as more general, abstract,
examples. Let us now illustrate that also noncommutative algebras fit
into the framework we have developed.

Let $\A$ be a vector space of (complex or real) matrices of dimension
$N$, and let $A^\dagger$ denote the hermitian transpose of the matrix
$A$. The Frobenius norm
\begin{equation}  \label{eq-Frobenius-norm}
  \normF{A} = \sqrt{\tr A^\dagger A}
\end{equation}
is strictly convex and, since $\A$ is finite-dimensional, the space
$(\A,\normF{\cdot})$ is a uniformly convex Banach space. 
As in Section~\ref{sec:ugs.from.vector.valued} 
(with $V=\Vt=W=\Wt=\A$), any linear map
$F:\A\to\A$ (which is automatically \emph{continuous}  since $\A$ is
finite-dimensional) induces a gradient relation $R_F$ on $\A$.
In particular, one may choose an inner derivation
\begin{align*}
  F(A) = [A,\delta]:= A\delta-\delta A
\end{align*}
for some (fixed) $\delta\in \A$. With such a choice, 
$\U=(\A,\A,\A,\A,R_F)$ becomes a
gradient space.

Moreover, via positive matrices one may introduce a linear ordering.
Namely, for matrices $A$ and $B$ one writes $A\geq B$ if the matrix
$A-B$ is positive definite or positive semidefinite.  With respect to
this ordering, $\U$ is an ordered gradient space.

For finite-dimensional matrices one may easily illustrate the fact that
$\max(\psi_1,\psi_2)$ is not necessarily comparable to every upper
bound of $\psi_1$ and $\psi_2$. By setting 
\begin{align*}
  \psi_1 = 
  \begin{pmatrix}
    1 & 0 \\ 0 & 0
  \end{pmatrix}\quad\text{and}\quad
  \psi_2 =
  \begin{pmatrix}
    0 & 0 \\ 0 & 1
  \end{pmatrix}
\end{align*}
it is straightforward, but somewhat tedious, to check that the $2\times 2$ identity matrix
$\mathbf{1}_2$ minimizes the Frobenius norm among hermitian matrices satisfying
$A-\psi_1\geq 0$ and $A-\psi_2\geq 0$. Setting
\begin{align*}
  A = 
  \begin{pmatrix}
    3 & \sqrt{5} \\ \sqrt{5} & 3
  \end{pmatrix}
\end{align*}
one readily checks that $A\geq\psi_1$, $A\geq\psi_2$ and that
$A-\mathbf{1}_2$ has the eigenvalues
\begin{align*}
  &\lambda_{+} = 2+\sqrt{5} > 0\\
  &\lambda_{-} = 2-\sqrt{5} < 0.
\end{align*}
Hence, $A$ is an upper bound for $\psi_1$ and $\psi_2$ which is not
comparable to $\mathbf{1}_2$. 

For self-adjoint algebras of bounded operators on a Hilbert space
(which, in some sense, are natural noncommutative examples of our
theory), the above situation is more or less generic, since if the
hermitian elements of such an algebra form a lattice, then the algebra
is commutative \cite{s:orderOperatorAlgebras}.

\subsection{Trace ideals}\label{sec:trace.ideals}

\noindent As concrete
infinite-dimensional noncommutative examples of our
framework, we consider trace ideals of compact operators (see e.g.\
Simon~\cite{s:traceIdeals} for a comprehensive treatment). Let $\B(H)$
denote the algebra of bounded linear operators on a separable Hilbert
space $H$; the scalar product on $H$ will be denoted by
$(\cdot,\cdot)$, and the induced norm by $\normH{\cdot}$. For
$A\in\B(H)$ one defines the adjoint operator $A^\ast$ via
$(Ax,y)=(x,A^\ast y)$, as well as the operator norm, given by
\begin{align*}
  \norm{A} = \inf_{\normH{x}=1}\normH{Ax}.
\end{align*}
A self-adjoint operator $A\in\B(H)$ (i.e.\ an operator satisfying
$A^\ast=A$) is called \emph{positive}, and one writes $A\geq 0$, if
$(Ax,x)\geq 0$ for all $x\in H$. For two operators $A,B\in\B(H)$ one
writes $A\geq B$ if $A-B\geq 0$.  The operator norm and the adjoint
operation makes $\B(H)$ into a $C^\ast$-algebra and, in particular, a
Banach space.  However, this is not the Banach space we want to use
for our purposes, since if $\B(H)$ is reflexive as a Banach space with
respect to the operator norm, then it is finite-dimensional (see e.g.\
Takesaki~\cite[p.~54]{t:oaI}). Instead, we shall consider trace ideals
together with their associated trace norms.  

It is a well-known fact that every nonempty proper ideal
$I\varsubsetneq \B(H)$ is contained in the ideal of compact operators
$\K(H)$ (\cite[Proposition~2.1]{s:traceIdeals}). Furthermore, every
compact operator $A\in\K(H)$ has a unique (up to ordering) sequence
$\{\mu_i(A)\}_{i\in I}$ of positive numbers (where $I$ is
either a finite or countable index set), called
\emph{singular values}, such that for $x\in H$,
\begin{align*}
  Ax = \sum_{i\in I}\mu_i(A)(e_i,x)\tilde{e}_i,
\end{align*}
where both $\{e_i\}_{i\in I}$ and $\{\tilde{e}_i\}_{i\in I}$ are
orthonormal sets \cite[Proposition~1.4]{s:traceIdeals}. In fact,
$\{\mu_i(A)\}_{i\in I}$ are the (strictly) positive eigenvalues of the self-adjoint
positive operator $|A|=\sqrt{A^\ast A}$.  Let $\{e_k\}_{k=1}^\infty$
be an orthonormal basis for $H$. The trace of an operator $A\geq 0$ is
defined as the sum
\begin{align*}
  \tr A = \sum_{k=1}^\infty(Ae_k,e_k)=\sum_{i\in I}\mu_i(A),
\end{align*}
which may or may not converge (in case it does, the sums are absolutely
convergent since every term is nonnegative).  For $1<p<\infty$, we let
$L^p(\B)$ denote the set of compact operators $A$ for which the sum
\begin{align*}
  \normp{A} := (\tr|A|^p)^{1/p} = \parac{\sum_{i\in I}\mu_i(A)^p}^{1/p}
\end{align*}
converges. (The case $p=2$ is an infinite-dimensional analogue of the
Frobenius norm~\eqref{eq-Frobenius-norm}.) Recall that $|A|^p$ is
defined via the functional calculus for operators (see
e.g.\ Rudin~\cite[Chapter~12]{r:functionalAnalysis}). The normed space
$(L^p(\B),\normp{\cdot})$ is a Banach space, and it is a particular
example of a \emph{symmetrically normed ideal}, which
satisfies
\begin{align}\label{eq:sni.properties}
  &\normp{BA}\leq \norm{B}\normp{A}\quad\text{and}\quad
  \norm{A}\leq\normp{A}
\end{align}
for $A\in L^p(\B)$ and $B\in\B(H)$ (\cite[Proposition~2.7]{s:traceIdeals}). 
Moreover,  
$L^p(\B)$ is known
to be uniformly convex (Dixmier~\cite[p.~30]{d:formesLineaires} and
McCarthy~\cite[Theorem~2.7]{m:cp}).  (Note that the above construction can be
generalized to traces on semifinite von Neumann algebras, see e.g.\
Takesaki~\cite[Chapter~IX]{t:oaII} for an introduction to
noncommutative integration theory.)  

Our aim is to illustrate that one can construct an ordered gradient
space with $\Vt=\Wt=\B(H)$ and $V=W=L^p(\B)$, with respect to the
standard ordering on operators as described above.  One may introduce
a gradient relation on $\B(H)\times\B(H)$ in many different ways,
e.g.\ as in Section~\ref{sec:ugs.from.continuous}, where one chooses a
bounded (and hence continuous) linear map $T:D\to\B(H)$ for a closed subset
$D\subseteq L^p(\B)$. Note that, in this context, a bounded operator
$T$ is such that there exists $C>0$ with
\begin{align*}
  \normp{T(A)}\leq C\normp{A}\quad\text{for all $A\in D$.}
\end{align*}
Consequently, by defining the gradient relation
\begin{align*}
  R_T = \{(A,T(A)):A\in D\},
\end{align*}
it follows that 
\begin{align}\label{eq:U.gradient.space.def}
  \U = (L^p(\B),\B(H),L^p(\B),\B(H),R_T)
\end{align}
is a gradient space. Furthermore, we shall prove that $\U$ is in fact an
ordered gradient space. The result is stated below in
Theorem~\ref{thm:BH.ordered.gradient.space}, but we postpone
the proof until the end of this section.

\begin{theorem}\label{thm:BH.ordered.gradient.space}
  Let $T:D\to L^p(\B)$ be a bounded linear map defined on a closed
  subset $D\subseteq L^p(\B)$ and let $1<p<\infty$. Then $\U
  = \paraa{L^p(\B),\B(H),L^p(\B),\B(H),R_T}$ is an ordered gradient
  space with respect to the standard partial ordering of positive
  operators.
\end{theorem}

\noindent
Poincar\'e sets, with respect to the gradient relation
$R_T$, are given by subsets $\Om\subseteq D$ such that $T$ is bounded
from below on $\Om$, i.e.\ there exists $c>0$ such that
\begin{align*}
  \normp{T(A)}\geq c\normp{A} \quad \text{for all } A\in \Om.
\end{align*}
Let us illustrate the fact that there are many operators $T:L^p(\B)\to
L^p(\B)$  which are bounded from below,
by the following example. For any $M\in\B(H)$ we
let $\widetilde{T}_M$ denote the multiplication operator induced by $M$, i.e
\begin{align*}
  \widetilde{T}_M(A) = MA
\end{align*}
for $A\in L^p(\B)$. (Recall that, since $L^p(\B)$ is an ideal in
$\B(H)$, the product $MA$ lies in $L^p(\B)$.) From
\eqref{eq:sni.properties} it follows that that the operator
$\widetilde{T}_M$ is bounded
\begin{align*}
  \|\widetilde{T}_M(A)\|_p=\normp{MA}\leq\norm{M}\normp{A},
\end{align*}
and we let $T_M$ denote the rescaled operator
$T_M=\widetilde{T}_M/2\norm{M}$, giving
\begin{align*}
  \normp{T_M(A)}\leq\tfrac{1}{2}\normp{A}.
\end{align*}
Finally, we let $T=\mathbf{1}-T_M$ (where $\mathbf{1}$ denotes
the identity operator on $L^p(\B)$) and deduce that
\begin{align*}
  \normp{T(A)}=\normp{A-T_M(A)}\geq\normp{A}-\normp{T_M(A)}
  \geq \tfrac{1}{2}\normp{A}.
\end{align*}
Hence, any subset of $L^p(\B)$ is a Poincar\'e set with respect to the
gradient relation defined by $T$. Further examples are given by
Fredholm operators, for which one may find natural Poincar\'e sets.

\begin{definition}
  Let $X$ and $Y$ be Banach spaces and let $F:X\to Y$ be a bounded
  linear operator. $F$ is a \emph{Fredholm operator} if
  \begin{enumerate}
  \item $\im F$ is closed,
  \item $\ker F$ is finite-dimensional,
  \item $\coker F=Y/\im F$ is finite-dimensional.
  \end{enumerate}
\end{definition}

\noindent
Next, we show that for a Fredholm operator, the complement of the
kernel is a Poincar\'e set.

\begin{proposition} \label{prop-Fredholm-imp-PI-set} 
    Let $X$ and $Y$ be Banach spaces,
and let $F:X\to Y$ be a Fredholm
  operator. Then there exist a closed subspace $V\subseteq X$ 
  and a constant $C>0$ such that $X=V\oplus \ker F$ and
  \begin{align*}
    \norm{v}_X\leq C\norm{F(v)}_Y 
   \quad \text{for all } v\in V.
  \end{align*}
\end{proposition}

\begin{proof}
  It is a standard fact that a finite-dimensional subspace of a normed
  space is complemented (see e.g.\ 
  Rudin~\cite[Lemma~4.21]{r:functionalAnalysis}), 
  i.e.\ there exists a closed subspace
  $V$ such that $X=V\oplus\ker F$. Hence, one may consider the
  operator $\Ftilde=F|_V:V\to\im F$ which is a bijective bounded
  operator between two Banach spaces (since $\im F$ is closed).  By
  the bounded inverse theorem \cite[Corollary~2.12\,(b)]{r:functionalAnalysis}, 
  there exists a bounded inverse
  $\Ftilde^{-1}:\im F\to V$. Thus, for every $u\in\im F$ there exists
  a constant $C>0$ such that
  \begin{align*}
    \|\Ftilde^{-1}(u)\|_X\leq C\norm{u}_Y. 
  \end{align*}
  In particular, one may choose $u=\Ftilde(v)$ (for arbitrary $v\in
  V$), which gives
  \begin{align*}
    \norm{v}_X\leq C\|\Ftilde(v)\|_Y=C\norm{F(v)}_Y
  \end{align*}
  and proves the second part of the statement.
\end{proof}

\noindent Finally, we will prove Theorem~\ref{thm:BH.ordered.gradient.space}, 
i.e.\ that $\U$ (as defined in
\eqref{eq:U.gradient.space.def}) is an ordered gradient space. The
results below are more or less standard, but we choose to repeat them
here for two reasons: firstly, statements in the literature are not
adapted to our particular setting and, secondly, we want to facilitate
for readers who are not so familiar with operator algebras. Let us
start by recalling the following lemma, which we state without proof.

\begin{lemma}[Lemma~2.6 in McCarthy~\cite{m:cp}]\label{lemma:mccarthy}
  For $0\le A,B\in \B(H)$ and $1\leq p<\infty$ it holds that
  \[ 
  \tr A^p + \tr B^p \leq \tr(A+B)^p.
  \] 
 In particular, if 
$A,B\in L^p(\B)$ and  $0  \le A  \le B$ then
\[
   \normp{B}^p - \normp{A}^p =    \tr B^p-\tr A^p\geq\tr(B-A)^p\geq 0,
\]
and thus $\normp{A} \le \normp{B}$.
\end{lemma}

\begin{lemma}\label{lemma:Ti.p.limit}
  If $\{A_i\}_{i=1}^\infty$ is a sequence of
  operators in $L^p(\B)$, then
  \begin{align*}
    \lim_{i\to\infty}\normp{A_i} = 0\implies
    \lim_{i\to\infty}(A_ix,x) = 0\text{ for all } x\in H.
  \end{align*}
\end{lemma}

\begin{proof}
  Since $\norm{A_i}\leq \normp{A_i}$, the Cauchy--Schwarz inequality
  yields
  \begin{align*}
    |(A_ix,x)|\leq \normH{A_ix}\normH{x}\leq\norm{A_i}\normH{x}^2
   \le \normp{A_i}\normH{x}^2,
  \end{align*}
  from which it follows that $|(A_ix,x)|\to 0$ as $i\to\infty$.
\end{proof}

\begin{lemma}\label{lemma:Lp.convergence.preserves.order}
  Let $\{A_i\}_{i=1}^\infty$ be a sequence in $L^p(\B)$ such that
  $A_i\geq B$ for some $B\in\B(H)$. If there exists $A\in L^p(\B)$
  such that $\lim_{i\to\infty}\normp{A-A_i}=0$ then $A\geq B$.
\end{lemma}

\begin{proof}
  Assume that $\normp{A-A_i}\to 0$ with $A_i\geq B$ for
  $i=1,2,\ldots$, which implies that
  \begin{align*}
    \lim_{i\to\infty}((A-A_i)x,x)= 0,
  \end{align*}
  by Lemma~\ref{lemma:Ti.p.limit}. Next, one may write
  \begin{align*}
    ((A-B)x,x) & 
    =((A-A_i)x,x) + ((A_i-B)x,x)
    \geq ((A-A_i)x,x),
  \end{align*}
  since $A_i\geq B$.  As $\lim_{i\to\infty}((A-A_i)x,x)=0$ it
  follows that $((A-B)x,x)\geq 0$ for all $x\in H$, which is
  equivalent to $A\geq B$.
\end{proof}

\begin{proof}[Proof of Theorem~\ref{thm:BH.ordered.gradient.space}]
This now follows directly from Lemmas~\ref{lemma:mccarthy}
and~\ref{lemma:Lp.convergence.preserves.order}.
\end{proof}

\subsection{Upper gradients of operator-valued functions}

In close analogy with the theory of \p-weak upper gradients in 
Section~\ref{sec:pweak.upper.gradient} one may introduce upper gradients of
operator-valued (or, more generally, Banach space valued) functions,
see Heinonen--Koskela--Shanmugalingam--Tyson~\cite{hkst:banachSpaceValued},
\cite{HKST}. In particular, we shall consider
functions from a metric measure space $(X,d,\mu)$ into the space
$L^p(\B)$, as introduced in Section~\ref{sec:trace.ideals}.

Thus, let $(X,d,\mu)$ be a metric measure space and let $S^{r,p}$
denote the set of ($\mu$-a.e.\ equivalence classes of) functions
$f:X\to L^p(\B)$ such that
\begin{align*}
  \normr{f} = \parac{\int_X\normp{f}^rd\mu}^{1/r} < \infty.
\end{align*}
Since $L^p(\B)$ is uniformly convex,
it is for certain values of $r$ and $p$ known that $S^{r,p}$ is uniformly
convex  (due to the fact that one
can prove a Clarkson inequality in $S^{r,p}$); for instance, one may
choose $1<p\leq 2$ and $r=p$ 
(see Takahashi--Kato~\cite{kt:ClarksonInequalities} for details).

For $q>1$, following \cite{hkst:banachSpaceValued}, we say that a
nonnegative measurable function $g:X\to\reals$ is a $q$-weak upper
gradient of $f\in S^{r,p}$ if
\begin{align}\label{eq:def.banach.upper.gradient}
  \normr{f(\gamma(a))-f(\gamma(b))}
  \leq \int_\gamma g \,ds
\end{align}
for $q$-almost every rectifiable curve $\gamma:[a,b]\to X$. We define
a gradient relation $R\subseteq S^{r,p}\times \Wt$, where $\Wt$
denotes the set of ($L^q(X)$-equivalence classes of) functions
$X\to\reals$, as follows
\begin{align*}
  R = \{(u,g):u\in S^{r,p},\,g\in\Wt\textrm{ and $g$ is a $q$-weak
    upper gradient of $u$}\}.
\end{align*}
(Or more precisely, there is a representative of $g$ which is a
$q$-weak upper gradient of some representative of $u$.)  As for the
weak upper gradients in Section~\ref{sec:pweak.upper.gradient}, it is
immediate to check that properties
\ref{g:sum} and \ref{g:mult.scalar}
are fulfilled, and so $R$ is indeed a gradient relation. Consequently,
we consider the gradient space
\begin{align*}
  \U = \paraa{S^{r,p},\Vt,L^q(X),\Wt,R},
\end{align*}
where $\Vt$ denotes the space of ($\mu$-a.e.\ equivalence classes) of functions
$X\to L^p(\B)$. Properties~\ref{gs:V.reflexive} and
\ref{gs:W.reflexive.convex}, of a gradient space, are fulfilled since
both $S^{r,p}$ and $L^q(X)$ are uniformly convex. Property
\ref{gs:minus.gradient} is satisfied as 
it follows immediately from
\eqref{eq:def.banach.upper.gradient} that if $g\in L^q(X)$ is a
$q$-weak upper gradient of $u$, then $g$ is also a $q$-weak upper
gradient of $-u$. The fact that \ref{gs:upper.gradient.limit} holds
follows (mutatis mutandis) from the corresponding result for upper
gradients considered in Section~\ref{sec:pweak.upper.gradient}
(cf.\ Bj\"orn--Bj\"orn~\cite[Proposition~2.3]{BBbook}). Hence, we conclude that $\U$ is
a gradient space.

The existence of Poincar\'e sets is analogous to the case of
real-valued functions in Section~\ref{sec:pweak.upper.gradient}.
Namely, 
Heinonen--Koskela--Shanmugalingam--Tyson~\cite[Theorem~4.3]{hkst:banachSpaceValued} 
shows that whether
a Poincar\'e inequality is supported or not, does not depend on
the Banach space in which the functions take their
values.  Therefore, if a metric measure space $X$ supports a
Poincar\'e inequality for real-valued functions, it supports a
Poincar\'e inequality for functions with values in an arbitrary Banach
space.

\end{document}